\documentclass{article}
\usepackage{CJK,CJKnumb,CJKulem,times,dsfont,ifthen,mathrsfs,latexsym,amsfonts,color}
\usepackage{amsmath,amsthm,makeidx,fontenc,amssymb,bm,graphicx,psfrag,listings,curves,extarrows,enumitem}
\usepackage[
linktocpage=true,colorlinks,citecolor=magenta,linkcolor=blue,urlcolor=magenta]{hyperref}

\usepackage{cite}
\usepackage{geometry}
\usepackage{indentfirst}

\usepackage[showonlyrefs]{mathtools}  
\mathtoolsset{showonlyrefs=true}

\usepackage{amssymb}
\makeatletter

\newcommand{\Rmnum}[1]{\expandafter\@slowromancap\romannumeral #1@}
\makeatother

\newtheorem{definition}{Definition}[section]
\newtheorem{theorem}{Theorem}[section]
\newtheorem{lemma}{Lemma}[section]
\newtheorem{corollary}{Corollary}[section]
\newtheorem{proposition}{Proposition}[section]
\newtheorem{remark}{Remark}[section]

\newcommand{\al}{\alpha}
\newcommand{\ga}{\gamma}
\newcommand{\Ga}{\Gamma}
\newcommand{\dl}{\Delta}
\newcommand{\e}{\varepsilon}

\newcommand{\iy}{\infty}

\newcommand{\la}{\lambda}
\newcommand{\vp}{\varphi}
\newcommand{\pa}{\partial}
\newcommand{\ti}{\tilde}

\newcommand{\vs}{\vskip .25in}

\newcommand{\ra}{\rightarrow}
\newcommand{\rh}{\rightharpoonup}

\newcommand{\lab}{\label}
\newcommand{\f}{\frac}
\newcommand{\bt}{\begin{theorem}}
	\newcommand{\et}{\end{theorem}}
\newcommand{\bl}{\begin{lemma}}
	\newcommand{\el}{\end{lemma}}
\newcommand{\bd}{\begin{definition}}
	\newcommand{\ed}{\end{definition}}
\newcommand{\bc}{\begin{corollary}}
	\newcommand{\ec}{\end{corollary}}
\newcommand{\bp}{\begin{proof}}
	\newcommand{\ep}{\end{proof}}
\newcommand{\bx}{\begin{example}}
	\newcommand{\ex}{\end{example}}
\newcommand{\bi}{\begin{exercise}}
	\newcommand{\ei}{\end{exercise}}
\newcommand{\br}{\begin{remark}}
	\newcommand{\er}{\end{remark}}
\newcommand{\be}{\begin{equation}}
	\newcommand{\ee}{\end{equation}}
\newcommand{\bal}{\begin{align}}
	\newcommand{\bn}{\begin{enumerate}}
		\newcommand{\en}{\end{enumerate}}
	\newcommand{\ba}{\begin{align}}
		\newcommand{\ea}{\begin{align}}
			\newcommand{\bg}{\begin{align*}}
				\newcommand{\eg}{\end{align*}}
			\newcommand{\bcs}{\begin{cases}}
				\newcommand{\ecs}{\end{cases}}

			\newcommand{\CR}{{\cal C}}
			\newcommand{\DR}{{\cal D}}
			\newcommand{\FR}{{\cal F}}

			\newcommand{\PR}{{\cal P}} 
			\newcommand{\SR}{{\cal S}} 
			\newcommand{\TR}{{\cal T}} 

			\newcommand{\C}{{\mathbb C}}
			\newcommand{\R}{{\mathbb R}}
			\newcommand{\RN}{{\mathbb R^N}}  
			\newcommand{\bean}{\begin{eqnarray*}}
				\newcommand{\eean}{\end{eqnarray*}}

			\newcommand{\sbr}[1]{\left(#1\right)}
			\newcommand{\mbr}[1]{\left[#1\right]}
			\newcommand{\lbr}[1]{\left\{#1\right\}}
			\newcommand{\rd}{\mathrm d}

			\newcommand{\s}{\star}
			\newcommand{\dx}{ ~\mathrm{d} x}

			\newcommand{\nm}[1]{\Vert #1 \Vert}
			\newcommand{\np}[1]{\Vert #1 \Vert_p}
			\newcommand{\nt}[1]{\Vert #1 \Vert_2}
			\newcommand{\ns}[1]{\Vert #1 \Vert_{2^*}}
			
			\setitemize{itemindent=38pt,leftmargin=0pt,itemsep=-0.4ex,listparindent=26pt,partopsep=0pt,parsep=0.5ex,topsep=-0.25ex}

			\numberwithin{equation}{section}
			
\begin{document}
				\theoremstyle{plain}

				\title{\bf  Normalized solutions for a class of Sobolev critical Schr\"odinger systems\thanks{ 	
					 E-mails: lhwmath@bimsa.cn (H. W. Li),  liuth19@mails.tsinghua.edu.cn (T. H. Liu),   zou-wm@mail.tsinghua.edu.cn (W. M. Zou)}}
				
				\date{}
				\author{
					{\bf Houwang Li$^1$,\; Tianhao Liu$^2$,\; Wenming Zou$^3$}\\
					\footnotesize	\it  1. Beijing Institute of Mathematical Sciences and Applications, Beijing, 100084, China.	\\
					\footnotesize \it 2.  Institute of Applied Physics and Computational Mathematics,  Beijing,  100088,  China.	\\ 
					\footnotesize \it 3.   Department of Mathematical Sciences, Tsinghua University, Beijing, 100084, China.
				}

				\maketitle
				\begin{center}
					\begin{minipage}{120mm}
						\begin{center}{\bf Abstract }\end{center}
						
						This paper focuses on the existence and multiplicity of normalized solutions for the following coupled  Schr\"odinger system with Sobolev critical coupling term:
						\begin{equation*}
							\left\{ 	\begin{aligned}
								&-\dl u+\la_1 u=\omega_1|u|^{p-2}u+\f{\al\nu}{2^*} |u|^{\al-2}|v|^\beta u,\quad \text{in }\RN,\\
								&-\dl v+\la_2 v=\omega_2|v|^{p-2}v+\f{\beta\nu}{2^*} |u|^\al |v|^{\beta-2}v,\quad \text{in }\RN,\\
								&\int_\RN u^2\dx =a^2,\int_\RN v^2\dx=b^2,
							\end{aligned}	\right.
						\end{equation*}
						where	$N\geq 3$, $a,b>0$,  $\omega_1,\omega_2, \nu\in\R\setminus\lbr{0}$,  and the exponents $p,\al,\beta$ satisfy
						\be 
	\al>1,~\beta>1,~~\al+\beta=2^*,~~2<p\leq 2^*={2N}/\sbr{N-2}.
						\ee
		The parameters $\lambda_1, \lambda_2\in \R$ will arise as  Lagrange multipliers	 that are not  prior given. This paper mainly presents several existence and multiplicity  results under explicit conditions on $a,b$ for the focusing case $\omega_1, \omega_2> 0$ and  attractive case $\nu >0$:
				\medbreak		
 	 	\quad (1) When $2<p<2+4/N$,	we prove that there exist two solutions: one is a local minimizer, which  serves as a normalized ground state, and the other is of mountain-pass type, which is a normalized excited state.
 	 	
						\vskip 0.03in
						\quad (2) When $2+4/N\leq p<2^*$, we prove that there exists   a mountain-pass type solution, which   serves as a normalized ground state.
						\vskip 0.03in
						\quad (3) When  $p = 2^*$, the existence and classification of normalized ground states are provided for and $N \geq 5$, alongside a non-existence result for $N = 3,4$.  These results reflect the properties of the Aubin-Talenti bubble, which attains the best Sobolev embedding constant.
						
					 \medbreak
						Furthermore,  we present  a  non-existence result for the defocusing case $\omega_1,\omega_2<0$.
						This paper, together with the paper [T. Bartsch,  H. W. Li and W. M. Zou.  
						{\it Calc. Var. Partial Differential Equations}  62 (2023) ],  provides a more comprehensive understanding of normalized solutions for Sobolev critical systems.      We believe our methods can also address the open problem of the multiplicity of normalized solutions for Schr\"odinger systems with Sobolev critical growth, with potential for future development and broader applicability.

						
						


						
						\vskip0.23in
						
					\end{minipage}
				\end{center}

				\vskip0.23in
				\section{Introduction and statement of results}
			 

			 \medbreak
		 The following Schr\"odinger  system describes the dynamics of  coupled nonlinear  waves 
			 	\be\label{equ1}
			 \left\{ 	\begin{aligned}
			 	 	&-i\f{\pa}{\pa t}\Phi_1=\dl \Phi_1+\omega_1 |\Phi_1|^{p-2}\Phi_1+\f{\al\nu}{\al+\beta} |\Phi_1|^{\al-2}|\Phi_2|^\beta\Phi_1,\\
			 		&-i\f{\pa}{\pa t}\Phi_2=\dl \Phi_2+\omega_2 |\Phi_2|^{p-2}\Phi_2+\f{\beta\nu}{\al+\beta}|\Phi_1|^{\al}|\Phi_2|^{\beta-2}\Phi_2,\\
			 		&\Phi_j=\Phi_j(x,t)\in\C,\quad(x,t)\in\RN\times\R,~j=1,2,\\
			 		 \end{aligned}	\right.
			 					 \ee 
		   where $i$ is the imaginary unit, and the parameters satisfy	 
			  \begin{equation}
			  	N\geq 3,  ~~ \omega_1,\omega_2, \nu\in\R\setminus\lbr{0}  ~\text{ and } ~ ~\al, \beta>1,~~   2< p, \alpha +\beta  \leq 2^*:={2N}/\sbr{N-2}. 
			  \end{equation}		  
			 System \eqref{equ1} arises in various physical models and has been extensively studied in recent years.
		For instance,  it appears in Hartree-Fock theory for a double condensate, specifically describing a binary mixture of Bose-Einstein condensates in two different hyperfine states.   In particular,  when $p=4,\al=\beta=2$, system \eqref{equ1} reduces to the well-known {\it Gross-Pitaevskii system},  which is widely applied in nonlinear optics. Moreover, certain models for ultracold quantum gases involve different exponents; see \cite{BG-2,BG-1,BG-3,Adhikari,Malomed}.
		For more background on system \eqref{equ1}, we refer to \cite{Peral etc=CVPDE=2009,Chen-Zou=CVPDE=2014,Chen-Zou=TransAMS=2015} and the references therein.
		\medbreak
From a physical perspective,  the solutions $\Phi_1$ and $\Phi_2$ represent  condensate amplitudes corresponding to different condensates, while the parameters $\omega_1$ and $\omega_2$  correspond to self-interactions within each component. These self-interactions are  called {\it focusing} when the sign is positive, and {\it defocusing} when negative.  The coupling constants $\alpha\nu /(\alpha+\beta)$ and $\beta\nu /(\alpha+\beta)$ define the strength and type of interaction between components $\Phi_1$ and $\Phi_2$.  The sign of $\nu$ whether the interaction between the two states is attractive or repulsive:  it is {\it attractive} when    $\nu>0$, and   {\it repulsive} when  $\nu<0$, indicating strong competition between the two states.

		 \medbreak 
			 
			   An important feature of system \eqref{equ1} is that any solution   satisfies the  conservation of mass, which plays a crucial role in the dynamics and stability of the system.  More precisely, the following $L^2$-norms
			  \begin{equation}
			  	\left\|\Phi_1(\cdot,t) \right\|_{L_x^2(\RN)}^2 ~ \text{ and } ~	\left\|\Phi_2(\cdot,t) \right\|_{L_x^2(\RN)}^2
			  \end{equation}
			  are independent of  time $t\in \R$. These norms have important physical significance:    in Bose–Einstein condensates,  the $L^2$-norms represent the number of particles of each component; in nonlinear optics framework,     they correspond to the power supply;			  
			and   in Hartree-Fock theory,
	they  represent the mass of condensate. 
			  
		 \medbreak
			 Among all  the solutions, the study of  {\it solitary waves}  is particularly important for system \eqref{equ1}, as the ground state  plays a crucial role in the dynamics of the solutions; for example, the determination of the threshold of the scattering
			 solutions, see \cite{Kenig2006,Kenig2008,Dodson,Duyckaerts,Holmer,Ibrahim}, and the universal profile of collapsing solutions for the critical nonlinear Schr\"odinger equations, see  \cite{Hmidi,Merle,Raphael} and the references therein.	
			 The  ansatz  $\Phi_1(x,t)=e^{i\la_1 t}u(x)$ and $\Phi_2(x,t)=e^{i\la_2t}v(x)$   for solitary wave solutions of  system \eqref{equ1} leads to the following  steady-state  coupled nonlinear  Schr\"odinger  system   
			 \be		 \label{equ2}
			 \left\{ 	\begin{aligned}
			 	&-\dl u+\la_1 u=\omega_1 |u|^{p-2}u+\f{\al\nu}{\al+\beta} |u|^{\al-2}|v|^\beta u,\quad \text{in }\RN,\\
			 	&-\dl v+\la_2 v=\omega_2 |v|^{p-2}v+\f{\beta\nu}{\al+\beta} |u|^\al |v|^{\beta-2}v,\quad \text{in }\RN,\\
			 	& \ \ u(x),  v(x)\to0~~\text{as }|x|\to\iy,
			 \end{aligned}	\right.
			 \ee

		For the study of  system \eqref{equ2},  there are   two distinct options concerning the frequency
		parameter $\lambda_1$ and $\lambda_2$, leading to two different research directions.   One  direction is to fix the parameters $\lambda_1, \lambda_2 > 0$. The existence and multiplicity of solutions to \eqref{equ2} have been extensively investigated over the past two decades. Numerous relevant studies can be found in the literature; we refer to \cite{Ambrosetti-Colorado=2007,Lin-Wei=CMP=2005,Sirakov=CMP=2007,
			Chen-Zou=CVPDE=2014,Chen-Zou=TransAMS=2015,Peral etc=CVPDE=2009,Bartsch-Dancer-Wang 2010,Bartsch-wang-2006,Bartsch2007} and the references therein.
		\medbreak
	In this paper, we focus on a different direction, specifically  investigating the solutions to system \eqref{equ2} that have a prescribed $L^2$-norm, which are 	commonly referred to as {\it normalized solutions}. In this direction, the parameters $\lambda_1, \lambda_2 \in \mathbb{R}$ cannot be prescribed but appear as multipliers with respect to the constraint   $L^2$-torus
		\be\label{masscon}
		\TR(a,b):=\Big\{(u,v)\in  H:\nt{u}^2=a^2,~\nt{v}^2=b^2 \Big\}.
		\ee
		where $H:=H^1(\RN)\times H^1(\RN)$ and $\np{ u}$ to denote the $L^p$-norm of $u$. 
		\medbreak
		From now on, we focus on the following normalized problem 
			\be\label{mainequ}
		\left\{ 	\begin{aligned}
			&-\dl u+\la_1 u=\omega_1|u|^{p-2}u+\f{\al\nu}{\alpha+\beta} |u|^{\al-2}|v|^\beta u,\quad \text{in }\RN,\\
			&-\dl v+\la_2 v=\omega_2|v|^{p-2}v+\f{\beta\nu}{\alpha+\beta} |u|^\al |v|^{\beta-2}v,\quad \text{in }\RN,\\
			&\int_\RN u^2\dx=a^2,\int_\RN v^2\dx=b^2,
		\end{aligned}	\right.
		\ee
		
From a variational point of view, 		 normalized solutions to problem \eqref{mainequ} can be obtained as critical points of the energy functional
		$I(u,v):  H^1(\RN)\times H^1(\RN) \to\R$,
		\be\label{energy}
		I(u,v):=\f{1}{2}\int_\RN \sbr{|\nabla u|^2+|\nabla v|^2} \dx
		-\f{1}{p}\int_{\RN} \sbr{|u|^p+|v|^p}\dx -\f{\nu}{\alpha+\beta}\int_{\RN}|u|^\al|v|^\beta\dx
		\ee
		on the constraint $	\TR(a,b)$
		with parameters $\la_1,\la_2\in\R$ appearing as Lagrange multipliers. Since $\al,\beta>1$, it is standard to conclude that $I(u,v)$ is of class $C^1$. Mathematicians and physicists typically focus on the solution with the least energy, as this often corresponds to good" behavior in physical terms. In this paper, we are particularly interested in the \textit{normalized ground state} that minimizes energy. Moreover, the existence of a \textit{normalized excited state} is also significant in physics.
		
	\begin{definition} {\rm A solution $(u_0,v_0)$ is  called a {\it normalized ground state } to system \eqref{mainequ} if  it satisfies
		\begin{equation}
			I|_{\TR (a,b)}'(u_0,v_0)=0,  \quad \text{ and }	 \quad 	I(u_0,v_0)=\inf\{I(u,v):~I|_{\TR(a,b)}'(u,v)=0, \ \ (u,v)\in \TR (a,b)\}.
		\end{equation}
		A solution $(u_0,v_0)$ is called  a {\it normalized excited state } to system \eqref{mainequ} if  it satisfies
		\begin{equation}
			I|_{\TR (a,b)}'(u_0,v_0)=0,  \quad \text{ and }	 \quad 	I(u_0,v_0)>\inf\{I(u,v):~I|_{\TR(a,b)}'(u,v)=0, \ \ (u,v)\in \TR (a,b)\}.
		\end{equation}	}
\end{definition}
\vskip 0.15in
The normalized ground state energy  $m_g(a,b)$ can be  defined by
\begin{equation*}
	m_g(a,b):=\inf_{(u,v)\in \TR (a,b)}I(u,v).
\end{equation*}
Moreover, the existence of   normalized ground state solution depends on    whether $m_g(a,b)$  can be attained.
From a mathematical perspective, new difficulties arise in the search for normalized solutions, making this problem particularly challenging.  For instance,  the
appearance of the constraint $\TR(a,b)$  makes lots of classical variation methods can not be applicable directly since $\la_1$, $\la_2$ are not given prior, and the existence of bounded
Palais–Smale sequences requires new arguments (the classical method  used to establish the boundedness of any Palais–Smale sequence for unconstrained Sobolev-subcritical problems fails to apply in this case). Moreover,  in the mass supercritical case where $p>2+\frac{4}{N}$ or $\alpha+\beta>2+\frac{4}{N}$, there are bounded Palais-Smale sequences that do not have a convergent
subsequence and converge weakly to $0$, since the  $L^2$-torus $\TR(a,b)$ is not a weak compact submanifold in $H$ (even in the radial case).

 \medbreak
    
System \eqref{mainequ}  can be viewed naturally as a counterpart to the following scalar equation, that is
\be\label{single1}
 \begin{aligned}
	&-\dl u+\la u=|u|^{p-2}u \quad \text{in}~\RN, \quad  \nt{u}^2=a^2.
\end{aligned}
\ee
which  has received a lot of attention in recent years. Over the past decades, many articles have investigated the existence of normalization solutions for equation \eqref{single1}  and systems \eqref{equ2},  proposing various methods to overcome the above difficulties. We refer the reader to \cite{Ikoma-Tanaka=AdvDE=2019,Jeanjean=1997,Jeanjean-Lu=CVPDE=2020,Jeanjean=2020,Soave=JDE=2020,Soave=JFA=2020,Wei-Wu=JFA=2022} for scalar equations,
\cite{Bartsch-Jeanjean=2018,Bartsch-Jeanjean-Soave=JMPA=2016,
	Bartsch-Soave=JFA=2017,
	Bartsch-Soave=CVPDE=2019,
	Bartsch-Zhong-Zou=MathAnn=2020,
	Gou-Jeanjean=2016,Gou-Jeanjean=2018,Li-Zou-1}
for  system of two equations, and \cite{Mederski-Cvpde2024,Schino-ADE2022} for system of $k$ equations.
However, these papers on  system  mainly focus on  the Sobolev subcritical case, and there are few results on  the Sobolev critical  system.
\medbreak
Recently, Bartsch, Li and Zou \cite{Bartsch-Li-Zou=2023} investigate the existence of normalized solutions to the system \eqref{mainequ} with a Sobolev critical nonlinearity $p=2^*$  and  a   subcritical coupling term $2<\al+\beta<2^*$ when $N=3,4$. They proposed the issue of multiplicity as an open problem in \cite[Remark 1.3]{Bartsch-Li-Zou=2023}.  To the best of our knowledge, there are no paper consider {\it existence,  multiplicity} of normalized solutions to system \eqref{mainequ}   under the following assumptions
\be\label{assumption1}
N\geq 3, ~~a,b >0 , ~~ \omega_1,\omega_2, \nu\in\R\setminus\lbr{0}  ~\text{ and }  ~2<p\leq 2^*, ~\al+\beta=2^*,~ \al,\beta>1.
\ee
The research we present here is a contribution to improve the picture of the  Sobolev critical situation.

 \medbreak

			First, for the  defocusing case $\omega_1<0$ and $\omega_2<0$, we observe that the following non-existence results.
				\begin{theorem}\label{thm5}
					Let $N\geq 3$, $\omega_1<0$ and $\omega_2<0$. Assuming that the exponents satisfy $2<p< 2^*$ and $\al+\beta=2^*$ with 
					\begin{equation}\label{f13}
						\begin{cases}
							\al>1,~\beta>1,  \quad \text{ if } \nu>0,\\
							\alpha\geq 2, ~\beta\geq 2,   \quad \text{ if } \nu<0,
						\end{cases}
					\end{equation}
				then	we have the following
					\begin{itemize}[fullwidth,itemindent=0em]
						\item[(1)]  if $N=3,4$, then the system \eqref{mainequ}  has no positive normalized solution $(u,v)\in H^1(\RN)\times H^1(\RN) $.
						\item[(2)]  if $N\geq 5$, then the system \eqref{mainequ}   has no positive normalized solution  $(u,v)\in H^1(\RN)\times H^1(\RN) $ satisfying the additional assumption that $u,v \in L^q(\RN) $ for some $q\in (0, N/(N-2)] $.
					\end{itemize}
				\end{theorem}
%
				\vskip 0.1in
	Therefore,  by Theorem \ref{thm5},   we will now focus on the focusing case  $\omega_1 > 0$ and $\omega_2 > 0$. Without loss of generality, we always assume $\omega_1 = \omega_2 = 1$ for simplicity.
				To present our results regarding the Sobolev critical system \eqref{mainequ}, we first need to introduce some established results in  \cite{Soave=JDE=2020} for the scalar equation \eqref{single1}, which play a crucial role in our argument.
				In variational approach, normalized solutions of \eqref{single1} are obtained as critical points
				of the associated energy functional 
					$$E(u):=\f{1}{2}\int_\RN|\nabla u|^2\dx-\f{1}{p}\int_\RN|u|^p\dx,\quad\text{for}~u\in H^1(\RN),$$
				on the constraint $S(a)= \lbr{ u\in H^1(\RN): \nt{u}^2=a^2}$. Moreover,  the   least energy is defined by
				\begin{equation}\label{defi of e(a)}
					e(a):=\inf_{u\in S(a)}E(u) =\inf\Big\{E(u):u\in H^1(\RN),~\nt{u}^2=a^2~\text{and}~\nt{\nabla u}^2=\ga_p\np{u}^p \Big\},
				\end{equation}
				where the parameter $\ga_p$ is denoted as 
				\be\label{sim1}
				\ga_p:=\f{N(p-2)}{2p}.  
				\ee		
			  We then have the following results, which are established in \cite{Soave=JDE=2020}.
				\newtheorem*{thmA}{Theorem A}
				\begin{thmA}\label{thmA}
					Let $N\ge3$, $p \in \sbr{2,2^*}\setminus \lbr{ 2+\f{4}{N}}$.
					Then up to a translation, scalar equation \eqref{single1} has a unique positive normalized 	solution $u_p$ with Lagrange multiplier $\la>0$, and 
					\begin{itemize}[fullwidth,itemindent=0em]
						\item[(1)]	if  $p<2+\f{4}{N}$, then
						\be
						e(a)=\inf_{\nt{u}^2=a^2}E(u)=E(u_p)<0;
						\ee
						\item[(2)]	if $p>2+\f{4}{N}$, then
						\be
						e(a)=\inf_{\nt{u}^2=a^2}\max_{t\in\R}E(e^{\f{N}{2}t}u(e^t\cdot))=\max_{t\in\R}E(e^{\f{N}{2}t}u_p(e^t\cdot))=E(u_p)>0;
						\ee
					\end{itemize}
				 Moreover,  in  both cases, the least energy $e(a)$ is strictly decreasing with respect to $a>0$.
				\end{thmA}
				\vskip0.1in
				From a variational point of view,  besides the Sobolev critical exponent  $2^*=2N/(N-2)$ for $N\geq 3$ and $2^*=+\infty$ for $N = 1, 2$, a new $L^2$-critical exponent  $\bar{p}=2+4/N$
				arises that plays a pivotal role in the study of normalized solutions to \eqref{mainequ}. This	threshold significantly impacts  the structure  of functional $I(u,v) $ (or $E(u)$, respectively) on the constraint  $\TR(a,b)$ (or $S(a)$, respectively),     and, consequently, influences the  choice of approaches when searching 
				for constrained critical points.
			 We   divide  our results  into three situation: (1) $L^2$-subcritical case $2<p<2+{4}/{N}$; (2) $L^2$-critical and $L^2$-supercritical (but Sobolev subcritical) case   $2+{4}/{N}\le p<2^*$;  (3)  Sobolev critical  case $p=2^*$.
				\subsection{\texorpdfstring{$L^2$}{}-subcritical \texorpdfstring{$2<p<2+\f{4}{N}$}{}}

				\medskip
			Now we  consider the case $\nu >0$ of attractive interaction.
					It is straightforward to verify that the functional $I(u,v)$  becomes  unbounded from below on $\TR (a,b)$ when  $\nu>0$.   To establish the existence of normalized solutions to \eqref{system1} in such cases, we draw inspiration from \cite{Jeanjean=2020,Soave=JFA=2020,Bartsch-Li-Zou=2023}. Specifically, solutions of \eqref{system1} satisfy the Pohozaev identity
					\be\lab{PHO}
					P(u,v):=\int_\RN \sbr{|\nabla u|^2+|\nabla v|^2} \dx
					-\ga_p\int_{\RN} \sbr{|u|^p+|v|^p}\dx -\nu\int_{\RN}|u|^\al|v|^\beta \dx=0.
					\ee
					Setting the Pohozaev manifold
						\be\label{pho}
					\PR(a,b):=\Big\{(u,v)\in \TR(a,b): P(u,v)=0 \Big\}.
					\ee
					Then   $\PR(a,b)\neq\emptyset$. The Pohozaev identity implies that any solution of \eqref{mainequ} belongs to $\PR(a,b)$, so that if $(u,v)$ is a minimizer of the following minimization problem
					\be\label{min}
					m(a,b):=\inf_{(u,v)\in\PR(a,b)} I(u,v).
					\ee
					then $(u,v)$ is a normalized ground state solution of \eqref{mainequ},  that is $m(a,b)=m_g(a,b)$.  By introducing a $L^2$-norm invariant transformation and   a suitable fiber map (detailed in Section \ref{sec1.4}), we decompose the Pohozaev manifold into three submanifolds:
					\[\PR(a,b)=\PR^+(a,b)\cup \PR^0(a,b)\cup \PR^-(a,b).\]
					If $2<p<2+4/N$ and $\nu>0$, the fiber map exhibits both convex and concave geometries under certain conditions, such as when the total mass  $a^2 + b^2$  is small. Such  specific geometric structure plays a  crucial role in searching for the multiplicity of normalized solutions. 
				\medbreak
					We now state that  the multiplicity results of normalized solutions to system \eqref{system1} for the $L^2$-subcritical  case. These solutions can be characterized as the local minimizer (normalized ground state) and the mountain pass solution (normalized excited state) of functional $I$ on the constraint $\PR(a,b)$ respectively. More precisely, we have the following results.
							
					\bt\label{thm1}
					Let $N\ge3$, $\nu>0$ and the exponents satisfy $\al>1, \beta>1, \alpha+\beta=2^*$ and
					\be\label{thm1con1}
					2<p<\begin{cases}2+\f{4}{N},&\quad\text{when}~N=3,4,\\ 2+\f{2}{N-2}, &\quad\text{when}~N\ge5. \end{cases}
					\ee
					There exists a constant $C_0>0$, such that if $\nu$ satisfies
					\be\label{Hsubcritical}
					0<\nu^{\f{2-p\ga_p}{2^*-2}}<C_0\sbr{a^{p(1-\ga_p)}+b^{p(1-\ga_p)}}^{-1},
					\ee
					then we have
					\begin{itemize}[fullwidth,itemindent=1.5em]
						\item[(1)]	$I|_{\TR(a,b)}$ has a critical point of  local minimum type $(u_+,v_+)$ at negative level $$m(a,b)=\inf_{\PR^+(a,b)} I(u,v)<0,$$ which is also a normalized
						ground state of  system  \eqref{mainequ}.
						\item[(2)]   $I|_{\TR(a,b)}$ has a second critical point of mountain pass type $(u_-,v_-)$ at positive level
						$$l(a,b)=\inf_{\PR^-(a,b)} I(u,v)>0,$$
					 which is also a normalized
				excited state of system \eqref{mainequ}
						\item[(3)]	Both components of $(u_+,v_+)$ and $(u_-,v_-)$ are positive functions, are radially symmetric, and solve system \eqref{mainequ} for
						suitable $\la_{1,+},\la_{2,+}>0$ and $\la_{1,-},\la_{2,-}>0$.
					\end{itemize}
					\et
					\br{ \rm  Actually, we can  give the explicit expression of $C_0$,
						$$C_0=\min\lbr{\f{1}{\ga_p},~\f{p(2^*+2-p\ga_p)}{22^*}}
						\f{(2^*-2)(2-p\ga_p)^{\f{2-p\ga_p}{2^*-2}}\SR_{\al,\beta}^{\f{2^*(2-p\ga_p)}{2(2^*-2)}}  }
						{\CR(N,p)(2^*-p\ga_p)^{\f{2^*-p\ga_p}{2^*-2}}},$$
						where $\SR_{\al,\beta}$ is defined by \eqref{Sobolev2}, and $\CR(N,p)$ is defined by \eqref{GN}. }
					\er
					
					To our best knowledge, Theorem \ref{thm1} appears to be the first multiplicity result of normalized solutions to Sobolev critical system \eqref{mainequ}. In this part, we outline the   ideas for establishing  the key multiplicity  of   solutions. 
					
					On the one hand, in order to obtain  the local minimum type solution  $(u_+,v_+)$  of the functional $I(u,v)$, we first compare the ground state $m(a,b)$ with the local minimum energy (see Lemma \ref{lemma 2.2}), namely:
					\begin{equation}
						 m(a,b)=\inf_{A_{R_0}(a,b)} I(u,v).
					\end{equation}
					     We then  establish an upper estimate of the  normalized ground state level $m(a,b)< \min \lbr{e(a),e(b)}<0$, which  allows us to recover the compactness  of Palais-Smale sequence at level $m(a,b)$ and obtain a local minimum type normalized ground state.
					
				On the other hand, since the functional $I(u,v)$ is unbounded from below on $\TR(a,b)$,    the structure of the functional suggests that there may exist  another normalized solution at the mountain pass level. 
				  To obtain the  second solution, we need to construct a mountain pass   level $\sigma_1$ (see \eqref{s1}), and   find a Palais-Smale sequence at such level by  applying the   Ghoussoub min-max principle described in \cite[Section 5]{Ghoussoub=1993}. Then using the Pohozaev identity, we can show that  such Palais-Smale sequence is   bounded in $H$.  However,    the weak limit of Palais-Smale sequence in $H$ may potentially possess vanishing components. To overcome this difficulty, a precise threshold for the mountain pass level is required to ensure that the weak limit for the Palais-Smale sequence has two nontrivial components.  More precisely, the mountain pass level  $\sigma_1$ must   be smaller than $m(a,b)+\f{\nu^{-(N-2)/2}}{N}\SR_{\al,\beta}^{N/2}$, where $\SR_{\al,\beta}$  is defined by \eqref{Sobolev2}.  In this paper, the energy estimate    can be established by observing the following result  (see Lemma \ref{lemma2.3}),   
					\begin{equation}
						\sigma_1= \inf_{\PR^+(a,b)} I(u,v):=l(a,b).
					\end{equation}
					The proof of energy estimate  \eqref{s2} depends on the choice of  suitable test function; see   Lemma \ref{keyest} and \ref{keyest2}. Our approach is inspired by the pioneering work \cite{Jeanjean=2020, Wei-Wu=JFA=2022}, which focuses on the multiplicity of normalized solutions to the following scalar equation with combined nonlinearity\be 
					\begin{aligned}
						& -\Delta u+\lambda u=\mu |u|^{q-2}u+|u|^{2^*-2}u \quad \text{in}~\RN, \quad  \nt{u}^2=a^2.
					\end{aligned} 
					\ee
					  However, due to the presence of coupling terms, the geometry of the Pohozaev manifold for system \eqref{mainequ} is more complicated than that for \eqref{single1}, which means that the methods used in \cite[Propositions 1.15 and 1.16]{Jeanjean=2020} and \cite[Lemma 3.1]{Wei-Wu=JFA=2022} cannot be applied directly here. Thus, we require more careful calculations in this paper.
					
%
%

					\vskip0.2in
					For the $L^2$-subcritical case $p<2+\f{4}{N}$, we find that the Sobolev critical system \eqref{mainequ} under assumption \eqref{assumption1} is more interesting than other cases. This is reflected not only in the multiplicity result for the attractive case 
					$\nu>0$ in Theorem \ref{thm1},  which reveals a complex structure of the corresponding functional  $I(u,v)$,  but also in the following existence result for the repulsive case  $\nu<0$. This contrasts with the non-existence result for $\nu<0$   presented in \cite{Bartsch-Li-Zou=2023} for the system \eqref{mainequ}     with a Sobolev critical nonlinearity $p=2^*$  and  a   subcritical coupling term $2<\al+\beta<2^*$.
					\medbreak
					Our next existence result can be stated as follows.
					\begin{theorem}\label{thm3}
						Let $N=3$ and $\nu<0$. Assuming that the exponents satisfy
						\be\lab{tem-101} 2<p<2+\frac{4}{N}, \text{ and }~	 p \leq 	\alpha,\beta < 2^*, \alpha+\beta=2^*,\ee
						we get that $I|_{\TR_r(a,b)}$ has a critical point  $(\bar u, \bar v)$ at negative level
						\begin{equation}
						 m_r(a,b):=\inf_{(u,v)\in\PR_r(a,b)} I(u,v)<0,
						\end{equation}where	$\TR_r(a,b)=\TR(a,b)\cap H_{rad}$ and $\PR_r(a,b):=\PR(a,b)\cap H_{rad}$  ($H_{rad}$ is  the subspace of radial symmetric functions in $H$).   Moreover, $(\bar u, \bar v)$ is   a       normalized  solution  with least energy among all radially symmetric normalized solution  of \eqref{mainequ} with Lagrange multipliers  $\bar{\lambda}_1,\bar{\lambda}_2>0$.
					\end{theorem}
					
%
%
%

					\subsection{\texorpdfstring{$L^2$}{}-critical and \texorpdfstring{$L^2$}{}-supercritical \texorpdfstring{$2+\f{4}{N}\le p<2^*$}{}}
					Now we focus on the case $2+\f{4}{N}\le p<2^*$ and $\nu>0$. Our results are as follows.
			 
					\bt\label{thm2}
					Let $N=3,4$, $\nu>0$, and the exponents satisfy
					$$\al>1,~\beta>1,~\alpha+\beta=2^*,~2+\f{4}{N}\le p<2^*.$$
					We further assume that
					$$a^{4/N}+b^{4/N}<(1+\f{2}{N})\CR(N,p), \quad\text{when}~p=2+\f{4}{N}.$$
					where $\CR(N,p)$ is the best constant in Gagliardo-Nirenberg inequality \eqref{GN}.
					If one of the following conditions holds:
					\begin{itemize}[fullwidth,itemindent=1em]
						\item[(C1)]	there exists $a_0>0$  such that   $2+\f{4}{N}<p<2^*$ and     $a, b\le a_0$;
						\item[(C2)]	 there exists $\nu_1=\nu_1(a,b,\alpha,\beta)>0$ such that	$\nu>\nu_1~$;
						\item[(C3)]	$a\le b$ and $\al<2$;
						\item[(C4)]	$b\le a$ and $\beta<2$,
					\end{itemize}
					then 	$I|_{\TR(a,b)}$ has a critical point of mountain pass type  $(\hat{u}_-,\hat{v}_-)$ at positive level $m(a,b)>0$, which is also a normalized
					ground state of \eqref{mainequ} with $\hat{\lambda}_1,\hat{\lambda}_2>0$.
					\et

					\br
					{\rm Actually, when $2+\f{4}{N}<p<2^*$, for fixed $\nu>0$, the constant  $a_0$ is defined by
						\be \label{defi of a0}
						a_0=a_0(\nu):=\mbr{ \f{1}{\ga_p\CR(n,p)}\sbr{ \f{2p\ga_p\nu^{-(N-2)/2}}{N(p\ga_p-2)}\SR_{\al,\beta}^{N/2}}
							^{\f{2-p\ga_p}{2}} }^{\f{1}{p(1-\ga_p)}}.  
						\ee
						By the properties of $e(a)$ in Theorem A, $a_0$ is actually choosed to satisfy
						\[e(a_0)=\f{\nu^{-(N-2)/2}}{N}\SR_{\al,\beta}^{N/2}.\]
						Therefore, under condition $(C1)$, we obtain $e(a)$, $e(b)\geq e(a_0)$. Furthermore, by Lemma \ref{est3-2}, it can be observed that under conditions $(C1)$-$(C4)$, there holds $m(a,b)<\min\lbr{e(a),e(b)}$.
						
					}	
					\er

					\br
					{\rm    Since $\TR(a,b)$ is not a weak compact submanifold in $H$,  a crucial step in recovering the the compactness of Palais-Smale is to prove that the Lagrange multipliers $\lambda_{1 }>0,\la_2>0$. As shown in Proposition \ref{PS1}, this step  can be achieved through a contradiction argument and Liouville-type theorems (see \cite[Lemma A.2]{Ikoma} and \cite[Theorem 8.4]{Souplet}), which are applicable only for $N=3,4$  when  $2+\f{4}{N}\le p<2^*$.   This is the only  reason we   consider  the cases  $N=3,4$  in Theorem \ref{thm2}. Finally,  we would like to emphasize that the key energy estimates in Lemmas \ref{est3-1} and \ref{est3-2} remain valid for 
					  $N\ge5$.
					 }
					\er

					\subsection{  Sobolev critical \texorpdfstring{$p=2^*$}{}}
					Our previous results primarily address the  case $2<p<2^*$ and $ \al+\beta=2^*$. We recall that the existence of normalized ground state solution to \eqref{mainequ} with the case  $p=2^*$ and $2<\al+\beta<2^*$ has already been investigated in \cite{Bartsch-Li-Zou=2023}. Moreover,     the case where $2<p<2^*$ and $2<\al+\beta<2^*$ has  been  thoroughly studied  in \cite{Li-Zou-1}.   A natural question is what happens    in the fully Sobolev critical case    $p=\alpha+\beta=2^*$. Does a normalized solution still exist in this particular case? We will provide an answer to this problem in this subsection.
					\medbreak
				  Before proceeding, we introduce some results concerning   the following coupled Sobolev critical system
					\be \label{system1}
					\left\{ 	\begin{aligned}
						&-\dl u =|u|^{2^*-2}u+\f{\al\nu}{2^*} |u|^{\al-2}|v|^\beta u,\quad \text{in }\RN,\\
						&-\dl v =|v|^{2^*-2}v+\f{\beta\nu}{2^*} |u|^\al |v|^{\beta-2}v,\quad \text{in }\RN,\\
						& u,v \in \mathcal{D}^{1,2}(\RN), \quad N\geq 3.
					\end{aligned}	\right.
					\ee
					Such system has been widely studied in these years, we refer to \cite{Clapp=DCDS=2019,Clapp=CVPDE=2018,ClappSzulkin,Guo=JDE=,HeYang2018} and references therein.
					%
					Define 	the Sobolev space $\mathcal{D}:=\mathcal{D}^{1,2}(\RN) \times \mathcal{D}^{1,2}(\RN)$ and the Nehari manifold
					\begin{equation}
						\mathcal{N}:=\lbr{ (u,v)\in\mathcal{D} \setminus \left\lbrace (0,0)\right\rbrace : \nt{\nabla u}^2+\nt{\nabla v}^2
							= \left\| u\right\|_{2^*}^{2^*}+\left\| v\right\|_{2^*}^{2^*}    +\nu \int_{\RN}|u|^\al|v|^\beta \dx}.
					\end{equation}
					It is easy to see that any nontrivial solution  of \eqref{system1} belongs to the Nehari manifold $\mathcal{N}$.
				A solution $(u_0,v_0)$ is called {\it ground state solution} if it has least energy among all nontrivial solutions, that is,  	\begin{equation} \label{defi of C}
						I(u_0,v_0)=\inf_{(u,v)\in \mathcal{N}} I(u,v)=:\mathcal{C}~.
					\end{equation}
					Moreover, in a standard way, there holds
					\begin{equation} \label{defi of C t}
						\mathcal{C}=\inf_{(u,v)\in\mathcal{D} \setminus \left\lbrace (0,0)\right\rbrace} \frac{1}{N}\frac{ \sbr{\nt{\nabla u}^2+\nt{\nabla v}^2}^{\frac{2^*}{2^*-2}}}{\sbr{ \left\| u\right\|_{2^*}^{2^*}+\left\| v\right\|_{2^*}^{2^*}    +\nu \int_{\RN}|u|^\al|v|^\beta}^{\frac{2
								}{2^*-2}}     }.
					\end{equation}
					The existence and classification results for ground state solution  were shown in \cite{HeYang2018}.
					Following the notations in \cite{HeYang2018}, we define the polynomial $F:\R^2\to \R$ by
					\begin{equation}
						F(x_1,x_2)=|x_1|^{2^*}+ |x_2|^{2^*}+ \nu |x_1|^{\al} |x_2|^\beta,
					\end{equation}		
					and denote by $\mathcal{X}$ the set of solutions to the maximization problem
					\begin{equation}
						F(\tilde{x}_1,\tilde{x}_2)	= F_{max}:= \max_{x_1^2+x_2^2=1} F(x_1,x_2), ~\text{ with  } ~\tilde{x}_1^2+ \tilde{x}_2^2=1.
					\end{equation}
					Then it follows from  \cite{HeYang2018} that  the system \eqref{system1} has a ground state solution of the form
					\begin{equation} \label{f15}
						\sbr{\widetilde{U}_{\e,y},\widetilde{V}_{\e,y}}:=\sbr{ \tilde{x}_1 F_{max}^{-\frac{N-2}{4}}U_{\e,y}, \tilde{x}_2 F_{max}^{-\frac{N-2}{4}}U_{\e,y}},\quad  \text{ where} ~ y\in \R^N,~ (\tilde{x}_1,\tilde{x}_2) \in \mathcal{X},~\varepsilon>0,
					\end{equation}
					and $U_{\e,y}$  is the Aubin-Talenti bubble, see \eqref{soliton} ahead. Therefore it follows from  \eqref{defi of C t} that
					\begin{equation}\label{f22}
						\mathcal{C}=I(\widetilde{U}_{\e,y},\widetilde{V}_{\e,y})=\frac{1}{N}F_{max}^{-\frac{N-2}{2}}\mathcal{S}^{\frac{N}{2}}.			\end{equation}
				 Our  first existence results  for $N\geq 5$ are as follows.
					\begin{theorem} \label{thm6}
						Let $N\geq 5$, $\nu>0$ and the exponents satisfy $\al>1,\beta>1,\al+\beta=p=2^*$. We further assume that  $	\tilde{x}_1 \neq 0,~ \tilde{x}_2 \neq 0  $ and     the mass $a$, $b$ satisfy that $a|\tilde{x}_2|=b|\tilde{x}_1| $,	then
						\begin{equation}
							m(a,b)=\inf_{(u,v)\in\PR^-(a,b)} I(u,v)= \frac{1}{N}F_{max}^{-\frac{N-2}{2}}\mathcal{S}^{\frac{N}{2}}=\mathcal{C}.
						\end{equation}
						Moreover, system \eqref{mainequ}  has a  normalized ground state solution, given by $	\sbr{\widetilde{U}_{\e_0,y},\widetilde{V}_{\e_0,y}}$  defined in \eqref{f15}  ~for $y\in \RN$ and the unique choice of $\varepsilon_0>0$ such that $ \nt{\widetilde{U}_{\e_0,y}}=a $ and $\nt{\widetilde{V}_{\e_0,y}}=b $. Furthermore, the function $ \sbr{\widetilde{U}_{\e_0,y},\widetilde{V}_{\e_0,y}}$ solves \eqref{mainequ} with $\lambda_1=\lambda_2=0$.		
					\end{theorem}
					
					\begin{remark}
						
						{ \rm The assumptions $	\tilde{x}_1 \neq 0,~ \tilde{x}_2 \neq 0   $ guarantee  that  the system   \eqref{system1} possesses a ground state solution with two nontrivial components.  Particularly, for a special case $2\alpha=2\beta=2^*$,   Chen and Zou  \cite{Zou 2015} proved that there exists a constant $\tilde{\nu}>0$ such that the following nonlinear problem 
							\begin{equation}
								\left\{ 	\begin{aligned}
									&k^{2^*-2}+\frac{\nu}{2^*}k^{\frac{2^*}{2}-2}l^{\frac{2^*}{2}}=1,\\
									&\frac{\nu}{2^*}k^{\frac{2^*}{2}}l^{\frac{2^*}{2}-2}+l^{2^*-2}=1.
								\end{aligned}	\right.
							\end{equation}
							has a positive solution $(k_0,l_0)$   with $k_0>0$, $l_0>0$
							 if $\nu >\tilde{\nu}$. In this case,  the  ground state solution of \eqref{system1}  is unique of the explicit expression $\sbr{ k_0U_{\e,y},l_0U_{\e,y}}$.  Therefore, as   an application 
						    of Theorem \ref{thm6}, if we assume that   $al_0=bk_0$  then system \eqref{mainequ} has a normalized  ground state solution $ \sbr{k_0U_{\e,y},l_0U_{\e,y}} \in \TR(a,b)$ for $y\in \RN$ and suitable value of  $\varepsilon>0$.}
					\end{remark}
					 
					\vskip 0.1in
					For the special case $N=3,4$, we have the following non-existence result.
					\begin{theorem} \label{thm4}
						Let $N=3,4$, $\nu>0$ and the exponents satisfy $\al>1,\beta>1,\al+\beta=p=2^*$, then system \eqref{mainequ} has no positive normalized solution.  			
					\end{theorem}
					
					
					\begin{remark}
			{\rm		Theorem \ref{thm6} together with Theorem \ref{thm4} provide an answer to the question whether a normalized solution to   system \eqref{mainequ} still exist  in the Sobolev critical case    $p=\alpha+\beta=2^*$.  The distinction between $N=3,4$ and $N\geq 5$ is significant, as it reflects the fact that the Aubin-Talenti bubble $U_{\epsilon,y}$ belongs to $L^2(\RN)$ if and only if   $N\geq 5$, and this distinction plays a crucial role in the analysis.
					}
					\end{remark}

					\subsection{Notations}\label{sec1.4}
					Throughout this paper, we always use the notations $\np{ u}$ to denote the $L^p$-norm of $u$, and we  simply denote $H:=H^1(\RN)\times H^1(\RN)$. Let $H_{rad}$ be the subspace of radial symmetric functions in $H$.
					We use $A\sim B$ to represent $C_1 B\le A\le C_2 B$ for some positive constants $C_1,C_2>0$. We use ``$\to$'' and ``$\rightharpoonup$'' to denote the strong convergence and weak convergence in corresponding space respectively. 	The capital letter $C$ will appear as a constant which may vary from line to line.		
					
					\medskip
					{\bf Sobolev inequality.} Recall the best Sobolev embedding constant
					\be\label{Sobolev1}
					\SR:=\inf_{u\in\DR^{1,2}(\RN)\setminus\{0\}}\f{\nt{\nabla u}^2}{\ns{u}^2},
					\ee
					where $\DR^{1,2}(\RN)$ is the completion of $\CR_0^\iy(\RN)$ with respect to
					the norm $\nm{u}:=\sbr{\int_\RN|\nabla u|^2\dx}^{\f{1}{2}}$.
					It is well known that $\SR$ is achieved by $u$ if only if
					\be\label{soliton}
					u\in\lbr{U_{\e,y}(x):~U_{\e,y}(x)=\big(\f{\sqrt{N(N-2)}\e}{\e^2+|x-y|^2}\big)^{\f{N-2}{2}},\quad \e>0,~y\in\RN}.
					\ee
					Moreover,
					\be
					\nt{\nabla U_{\e,y}}^2=\ns{U_{\e,y}}^{2^*}=\SR^{\f{N}{2}}.
					\ee To simplify the notations, we denote $U_\e(x):=U_{\e,0}(x)$.
					Define
					\be\label{Sobolev2}
					\SR_{\al,\beta}:=\inf_{u,v\in\DR^{1,2}(\RN)\setminus\{0\}}\f{\int_\RN \sbr{|\nabla u|^2+|\nabla v|^2}\dx}
					{ \sbr{\int_\RN|u|^\al|v|^\beta\dx}^{2/2^*} },
					\ee
					then from \cite{NA2000} we know that
					$$\SR_{\al,\beta}=\sbr{\sbr{\f{\al}{\beta}}^{\beta/2^*}+\sbr{\f{\beta}{\al}}^{\al/2^*}}\SR,$$
					where $\SR$ is defined by \eqref{Sobolev1}.
					
					\medbreak
					
					{\bf Gagliardo-Nirenberg inequality.} Recall the Gagliardo-Nirenberg inequality
					\begin{equation}\lab{GN}
						\np{u}\leq \CR(N,p)\nt{u}^{1-\ga_{p}}\nt{\nabla u}^{\ga_{p}}, \quad \forall u\in H^1(\RN),
					\end{equation}
					where $\ga_{p}$ is defined by \eqref{sim1}, and $\CR(N,p)$ is the sharp constant satisfying
					\begin{equation}\label{f25}
						\CR(N,p)^{-1}=\inf_{u\in H^1(\RN)\setminus \lbr{0}} \cfrac{\nt{u}^{1-\ga_{p}}\nt{\nabla u}^{\ga_{p}} }{\np{u} }=(p\ga_{p})^{\frac{\ga_{p}}{2}}(1-p\ga_{p})^{\frac{1}{p}-\frac{\ga_{p}}{2}}\nt{Z}^{1-\frac{2}{p}}.
					\end{equation}
					Here $Z$ is the unique solution of
					\be
					\left\{ 	\begin{aligned}
						&-\dl Z+Z =|Z|^{p-2}Z\quad \text{in }\RN,\\
						& Z>0 ~\text{ and }~ Z(x) \to 0  \quad \text{ as } |x|\to \infty ,\\
						&  Z(0)=\max_{x\in\RN} Z(x).
					\end{aligned}	\right.
					\ee
					For more details, we refer to \cite{Weinstein=CMP}. Moreover,the function $Z^{\kappa,\rho}(x) := \kappa Z( \rho x)$ satisfies
					\begin{equation}
						-\dl Z^{\kappa,\rho}+\rho^2Z =\kappa^{p-2}\rho^2|Z^{\kappa,\rho}|^{p-2}Z^{\kappa,\rho}\quad \text{in }\RN,
					\end{equation}
					and \eqref{f25} is achieved by $u$ if and only if
					\begin{equation}
						u\in \lbr{  Z^{\kappa,\rho}(\cdot +y): \kappa>0,\rho>0, y\in \RN}.
					\end{equation}

					\medbreak
					{\bf An importmant transformation.} We introduce a $L^2$-norm invariant transformation $t\s u(x):=e^{\f{N}{2}t}u(e^tx)$ and
					$$(u,v)\in \TR(a,b)\mapsto t\s(u,v):=(t\s u,t\s v)\in \TR(a,b).$$
					We define the following fiber map
					\be\label{fiber}
					\begin{aligned}
						\Phi_{(u,v)}(t):=I(t\s(u,v))=&\f{1}{2}e^{2t}\int_\RN \sbr{|\nabla u|^2+|\nabla v|^2}\dx
						-\f{1}{p}e^{p\ga_pt}\int_{\RN} \sbr{|u|^p+|v|^p} \dx\\
						&-\f{\nu}{2^*}e^{2^*t}\int_{\RN}|u|^\al|v|^\beta\dx.
					\end{aligned}
					\ee
					By a direct computation, we observe that $\Phi_{(u,v)}'(t)=P(t\s(u,v))$, where $P(u,v)$ is defined by \eqref{PHO}. Therefore, it holds
					\begin{equation}
						\PR(a,b)=\lbr{(u,v)\in \TR(a,b): \Phi_{(u,v)}'(0)=0}.
					\end{equation}
					In this direction, we decompose $\PR(a,b)$ into three  disjoint  submanifolds $\PR(a,b)=\PR^+(a,b)\cup \PR^0(a,b)\cup \PR^-(a,b)$, which are  given by
					\begin{equation}
						\begin{aligned}
							&\PR^+(a,b):=\lbr{(u,v)\in \PR(a,b): \Phi_{(u,v)}''(0)>0},\\
							&\PR^0(a,b):=\lbr{(u,v)\in \PR(a,b): \Phi_{(u,v)}''(0)=0},\\
							&\PR^-(a,b):=\lbr{(u,v)\in \PR(a,b): \Phi_{(u,v)}''(0)<0}.\\
						\end{aligned}
					\end{equation}

					\subsection{Structure of the paper}
					In the remaining sections of this paper, we provide proofs for our main results. In Section \ref{Sect2}, we address the attractive case, where $\nu > 0$, and present the proofs for Theorem \ref{thm1} in Subsection \ref{Sect2.1} and for Theorem \ref{thm2} in Subsection \ref{Sect2.2}. Moving on to Section \ref{Sect3}, we consider the repulsive case $\nu < 0$, and provide the proof for Theorem \ref{thm3}. In Section \ref{Sect4}, we establish the non-existence result for the defocusing case, as stated in Theorem \ref{thm5}. Lastly, in Section \ref{Sect5}, we explore the Sobolev critical case with $p = \alpha + \beta = 2^*$ and present the proofs for both Theorem \ref{thm6} and Theorem \ref{thm4}

					\vskip0.23in
					\section{Existence for the attractive case \texorpdfstring{$\nu>0$}{}} \label{Sect2}
					In this section, we study the existence of normalized solution of \eqref{mainequ} for the  attractive case  $\nu>0$. Throughout this section, we always work under the assumptions \eqref{assumption1}.
					To obtain the compactness of Palais-Smale sequence, we need the following additional assumptions
					\be\label{H3-1}
					2<p<2^*~\text{when}~N=3,4;\quad 2<p<2+\f{2}{N-2}~\text{when}~N\ge5,
					\ee
					\be\label{H3-2}
					a^{4/N}+b^{4/N}<(1+\f{2}{N})\CR(N,p)\quad\text{when}~p=2+\f{4}{N}.
					\ee
					Now we give the following compactness result for $\nu>0$.
					\begin{proposition}\label{PS1}
						Assume that \eqref{H3-1}, \eqref{H3-2} hold and
						\be\label{monotonicity}
						m(a,b)\le m(a_1,b_1)\quad\text{for any}~0<a_1\le a,~0<b_1\le b.
						\ee
						Let $\{(u_n,v_n)\}\subset \TR(a,b)$ be a sequence consisting of radial symmetric functions such that
						\be\label{tem3-1}
						I'(u_n,v_n)+\la_{1,n}u_n+\la_{2,n}v_n\to0\quad \text{for some}~\la_{1,n},\la_{2,n}\in\R,
						\ee
						\be\label{tem3-2}
						I(u_n,v_n)\to c,\quad P(u_n,v_n)\to0,
						\ee
						\be\label{tem3-3}
						u_n^-,v_n^-\to0,~\text{a.e. in}~\RN,
						\ee
						with
						\be\label{levelcon1}
						c\neq0,\quad c\neq e(a),\quad c\neq e(b),
						\ee
						and
						\be\label{levelcon2}
						c<\f{\nu^{-(N-2)/2}}{N}\SR_{\al,\beta}^{N/2}+ \min \lbr{~0,~e(a),~e(b),~m(a,b)~}.
						\ee
						Then there exists $u,v\in H_{rad}^1(\RN)$, $u,v>0$ and $\la_1,\la_2>0$ such that
						up to a subsequence $(u_n,v_n)\to(u,v)$ in $H^1(\RN)\times H^1(\RN)$ and
						$(\la_{1,n},\la_{2,n})\to(\la_1,\la_2)$ in $\R^2$.
					\end{proposition}
					\bp
					The proof is divided into three steps.
					\begin{itemize}[fullwidth,itemindent=1em]
						\vskip 0.1in
						\item[Step 1.)]	We show that $\{(u_n,v_n)\}$ is bounded in $H^1(\RN)\times H^1(\RN)$.
						
						If $2<p<2+\f{4}{N}$, there is $p\ga_p<2$. Then using $P(u_n,v_n)\to0$, one can see from the  Gagliardo-Nirenberg inequality  that
						for $n$ large enough,
						\begin{align*}
							c+1&\ge I(u_n,v_n)-\f{1}{2^*}P(u_n,v_n)\\
							&=\f{1}{N}\sbr{ \nt{\nabla u_n}^2+\nt{\nabla v_n}^2} -\f{2^*-p}{2^*p}\sbr{\np{u_n}^p+\np{v_n}^p}\\
							&\ge \f{1}{N}\sbr{ \nt{\nabla u_n}^2+\nt{\nabla v_n}^2}-C\sbr{ \nt{\nabla u_n}^2+\nt{\nabla v_n}^2}^{p\ga_p/2},
						\end{align*}
						for some $C>0$, which implies that $\{(u_n,v_n)\}$ is bounded.
						Now if $2+\f{4}{N}\le p<2^*$, then $p\ga_p\ge2$ and for $n$ large enough
						\begin{align*}
							c+1&\ge I(u_n,v_n)-\f{1}{2}P(u_n,v_n)\\
							&=\f{p\ga_p-2}{2p}\sbr{\np{u_n}^p+\np{v_n}^p}+\f{\nu}{N}\int_\RN |u_n|^\al|v_n|^\beta\dx.
						\end{align*}
						Combining \eqref{H3-2} with the fact that $P(u_n,v_n)=o(1)$, we conclude that $\{(u_n,v_n)\}$ is bounded.
						Moreover, from
						$$\la_{1,n}=-\f{1}{a}I'(u_n,v_n)[(u_n,0)]+o(1)\quad\text{and}\quad
						\la_{2,n}=-\f{1}{b}I'(u_n,v_n)[(0,v_n)]+o(1), $$
						we know that $\la_{1,n},\la_{2,n}$ are also bounded.
						So there exists $u,v\in H^1(\RN)$, $\la_1,\la_2\in\R$ such that up to a subsequence
						$$(u_n,v_n)\rh(u,v)\quad\text{in} ~H^1(\RN)\times H^1(\RN),$$
						$$(u_n,v_n)\ra(u,v)\quad\text{in} ~L^q(\RN)\times L^q(\RN),~\text{for}~2<q<2^*,$$
						$$(u_n,v_n)\ra(u,v)\quad\text{a.e. in}~\RN,$$
						$$(\la_{1,n},\la_{2,n})\to(\la_1,\la_2)\quad\text{in} ~\R^2.$$
						Then \eqref{tem3-1} and \eqref{tem3-3} give that
						\be\label{tem3-4}
						\left\{ 	\begin{aligned}
							&I'(u,v)+\la_1 u+\la_2v=0,\\
							&u\ge0,v\ge0,
						\end{aligned}\right.
						\ee
						and hence $P(u,v)=0$.

						\vskip 0.1in
						\item[Step 2.)]	We show that the weak limit $u\neq0$ and $v\neq0$.
						
						Without loss of generality,  we assume that $u=0$ by contradiction. There are two cases.
						If $v=0$, then from $P(u_n,v_n)=o(1)$ we obtain that
						$$\nt{\nabla u_n}^2+\nt{\nabla v_n}^2=\nu\int_\RN |u_n|^\al|v_n|^\beta\dx+o(1)
						\le \nu\SR_{\al,\beta}^{-2^*/2}\sbr{ \nt{\nabla u_n}^2+\nt{\nabla v_n}^2 }^{2^*/2}+o(1).$$
						Assuming $\nt{\nabla u_n}^2+\nt{\nabla v_n}^2\to l\ge0$, we immediately conclude $l\le \nu\SR_{\al,\beta}^{-2^*/2} l^{-2^*/2}$,
						which gives that $l=0$ or $l\ge \nu^{-(N-2)/2}\SR_{\al,\beta}^{N/2}$.
						As a consequence $c=0$ or
						$$c=\lim_{n\to\iy} I(u_n,v_n)=\f{l}{N}\ge \f{\nu^{-(N-2)/2}}{N}\SR_{\al,\beta}^{N/2},$$
						and both of them are contradictions.
						Now if $v\neq0$, by maximum principle we know that
						\be
						\left\{ 	\begin{aligned}
							&-\dl v+\la_2 v=v^{p-1},\quad\text{in}~\RN, \\
							&v>0.
						\end{aligned}\right.
						\ee
						Using \eqref{H3-1}, \cite[Lemma A.2]{Ikoma} and \cite[Theorem 8.4]{Souplet}, we obtain $\la_2>0$.
						Let $\bar v_n=v_n-v$. Then
						$$\nt{\nabla v_n}^2=\nt{\nabla \bar v_n}^2+\nt{\nabla v}^2+o(1),$$
						$$\np{v_n}^p=\np{v}^p+o(1),$$
						$$\int_\RN |u_n|^\al|v_n|^\beta\dx=\int_\RN|u_n|^\al|\bar v_n|^\beta\dx+o(1).$$
						It follows that
						$$\begin{aligned}
							o(1)&=P(u_n,v_n)=\nt{\nabla u_n}^2+\nt{\nabla \bar v_n}^2-\nu\int_\RN|u_n|^\al|\bar v_n|^\beta\dx+P(u,v)+o(1)\\
							&=\nt{\nabla u_n}^2+\nt{\nabla \bar v_n}^2-\nu\int_\RN|u_n|^\al|\bar v_n|^\beta\dx+o(1).
						\end{aligned}$$
						Similar as before, there holds $\nt{\nabla u_n}^2+\nt{\nabla \bar v_n}^2\to0$ or
						$\liminf_{n\to\iy}\nt{\nabla u_n}^2+\nt{\nabla \bar v_n}^2\ge \nu^{-(N-2)/2}\SR_{\al,\beta}^{N/2}$.
						If $\nt{\nabla u_n}^2+\nt{\nabla \bar v_n}^2\to0$, i.e., $u_n,\bar v_n\to0$ in $\DR^{1,2}(\RN)$, then
						$$\begin{aligned}
							&\quad~~\nt{\nabla\bar v_n}^2+\la_2\nt{\bar v_n}^2\\
							&=\sbr{I'(u_n,v_n)+\la_{1,n}u_n+\la_{2,n}v_n}[(0,\bar v_n)]-\sbr{I'(u,v)+\la_1u+\la_{2}v}[(0,\bar v_n)]+o(1)\\
							&=o(1).\end{aligned}$$
						That is $v_n\to v$ in $H^1(\RN)$. As a result,
						$$c=\lim_{n\to\iy}I(u_n,v_n)=\lim_{n\to\iy}\f{1}{N}\sbr{\nt{\nabla u_n}^2+\nt{\nabla \bar v_n}^2}+E(v)=e(b),$$
						which is  a contradiction. On the other hand, if
						$\liminf_{n\to\iy}\nt{\nabla u_n}^2+\nt{\nabla \bar v_n}^2\ge \nu^{-(N-2)/2}\SR_{\al,\beta}^{N/2}$,
						we have again
						$$c\ge\lim_{n\to\iy}\f{1}{N}\sbr{\nt{\nabla u_n}^2+\nt{\nabla \bar v_n}^2}+e(\nt{v})
						\ge \f{\nu^{-(N-2)/2}}{N}\SR_{\al,\beta}^{N/2}+e(b),$$
						a contradiction.

						\vskip 0.1in
						\item[Step 3.)]	We show the strong convergence.

						Let $(\bar u_n,\bar v_n)=(u_n-u,v_n-v)$. Then
						$$o(1)=P(u_n,v_n)=\nt{\nabla \bar u_n}^2+\nt{\nabla \bar v_n}^2-\nu\int_\RN|\bar u_n|^\al|\bar v_n|^\beta\dx+o(1).$$
						Similar to before,  there are two cases
						$$\text{whether}\quad \nt{\nabla \bar u_n}^2+\nt{\nabla \bar v_n}^2\to0 \quad\text{or}\quad \liminf_{n\to\iy}\nt{\nabla\bar u_n}^2+\nt{\nabla \bar v_n}^2\ge \nu^{-(N-2)/2}\SR_{\al,\beta}^{N/2}.$$
						If the second case occur, then
						$$c\ge\lim_{n\to\iy}\f{1}{N}\sbr{\nt{\nabla \bar u_n}^2+\nt{\nabla \bar v_n}^2}+m(\nt{u},\nt{v})
						\ge \f{\nu^{-(N-2)/2}}{N}\SR_{\al,\beta}^{N/2}+m(a,b),$$
						which is a contradiction. So $\nt{\nabla \bar u_n}^2+\nt{\nabla \bar v_n}^2\to0$, i.e., $u_n,\bar v_n\to0$ in $\DR^{1,2}(\RN)$.
						Moreover by maximum principle, $(u,v)$ is a positive solution of \eqref{tem3-4},
						and from \cite[Lemma A.2]{Ikoma} and \cite[Theorem 8.4]{Souplet}, we obtain immediately $\la_1,\la_2>0$.
						Noting that
						$$\begin{aligned}
							&\quad~~\nt{\nabla\bar u_n}^2+\la_1\nt{\bar u_n}^2+\nt{\nabla\bar v_n}^2+\la_2\nt{\bar v_n}^2\\
							&=\sbr{I'(u_n,v_n)+\la_{1,n}u_n+\la_{2,n}v_n}[(\bar u_n,\bar v_n)]-\sbr{I'(u,v)+\la_1u+\la_{2}v}[(\bar u_n,\bar v_n)]+o(1)\\
							&=o(1),\end{aligned}$$
						we obtain $(u_n,v_n)\to(u,v)$ in $H^1(\RN)\times H^1(\RN)$. We complete the proof.
					\end{itemize}
					\ep

					\subsection{The   case \texorpdfstring{$2<p<2+\f{4}{N}$}{}} \label{Sect2.1}

					Recall the fiber map defined by \eqref{fiber}
					$$\Phi_{(u,v)}(t):=\f{1}{2}e^{2t}\int_\RN \sbr{|\nabla u|^2+|\nabla v|^2}\dx
					-\f{1}{p}e^{p\ga_pt}\int_{\RN} \sbr{|u|^p+|v|^p}\dx-\f{\nu}{2^*}e^{2^*t}\int_{\RN}|u|^\al|v|^\beta\dx.$$
					We define
					\be
					h(\rho):=\f{1}{2}\rho^2-A\rho^{p\ga_p}-B\rho^{2^*}
					\ee
					with
					$$A:=\f{1}{p}\CR(N,p)\sbr{a^{p(1-\ga_p)}+b^{p(1-\ga_p)}},\quad B:=\f{1}{2^*}\nu\SR_{\al,\beta}^{-2^*/2}.$$
					Then for any $u,v\in H^1(\RN)$,
					$$I(u,v)\ge h(\sbr{\nt{\nabla u}^2+\nt{\nabla v}^2}^{1/2}).$$
					Futher, there is
					$$h'(p)=\rho^{p\ga_p-1}\sbr{g(\rho)-p\ga_p A}$$
					with $g(\rho):=\rho^{2-p\ga_p}-2^*B\rho^{2^*-p\ga_p}$.
					Let
					\be\label{rho}
					\rho_*:=\sbr{\f{2-p\ga_p}{2^*(2^*-p\ga_p)B}}^{1/(2^*-2)}.
					\ee
					It is easy to check that
					$g(\rho)$ is strictly increasing in $(0,\rho_*)$ and is strictly decreasing in $(\rho_*,+\iy)$.
					By direct computations, assumption \eqref{Hsubcritical} gives that
					$g(\rho_*)>p\ga_pA$ and $h(\rho_*)>0$, which means that $h(\rho)$ has only two critical points
					$0<\rho_1<\rho_*<\rho_2$ with
					$$h(\rho_1)=\min_{0<\rho<\rho_*} h(\rho)<0,$$
					$$h(\rho_2)=\max_{\rho>0} h(\rho)>0.$$
					Moreover, there exist $R_0$ and $R_1$ such that $h(R_0)=h(R_1)=0$ and $h(\rho)>0$ iff $\rho\in(R_0,R_1)$.
					
					Using \eqref{Hsubcritical} again, we can also prove in  a standard way that $\PR^0(a,b)=\emptyset$, and
					$\PR(a,b)$ is a smooth manifold of codimension $2$ in $\TR(a,b)$, see \cite[Lemma 5.2]{Soave=JDE=2020}
					(or \cite{Bartsch-Jeanjean-Soave=JMPA=2016}) for more details. This fact can in turn be used in the following lemma.
					\bl\label{structure1}
					For every $(u,v)\in \TR(a,b)$, $\Phi_{(u,v)}(t)$ has exactly two critical points $t_+(u,v)<t_-(u,v)$ and
					two zeros $c(u,v)<d(u,v)$ with
					$t_+(u,v)<c(u,v)<t_-(u,v)<d(u,v)$. Moreover,
					\begin{itemize}[fullwidth,itemindent=0em]
						\item[(a)]	$t\s(u,v)\in\PR^+(a,b)$ if and only if $t=t_+(u,v)$; $t\s(u,v)\in\PR^-(a,b)$ if and only if $t=t_-(u,v)$,
						\item[(b)] 	$\sbr{\nt{\nabla t\s u}^2+\nt{\nabla t\s v}^2}^{1/2}\le R_0$ for every $t\le c(u,v)$ and
						\be
						I(t_+(u,v)\s(u,v))=\min\lbr{ I(t\s(u,v)): t\in\R~~\text{and}~~
							\sbr{\nt{\nabla t\s u}^2+\nt{\nabla t\s v}^2}^{1/2}\le R_0}<0,
						\ee
						\item[(c)] 	$\Phi_{(u,v)}(t)$ is strictly dereasing and concave on $(t_-(u,v),+\iy)$ and
						$$\Phi_{(u,v
							)}(t_-(u,v))=\max_{t\in\R}\Phi_{(u,v)}(t),$$
						\item[(d)] 	the maps $(u,v)\mapsto t_+(u,v)$ and $(u,v)\mapsto t_-(u,v)$ are of class $C^1$.
					\end{itemize}
					\el
					\begin{proof}
						The proof is completely analogue to the one in \cite[Lemma 5.3]{Soave=JDE=2020}, so we omit the details.
					\end{proof}
					
					For any $R>0$, let
					$$A_R(a,b):=\lbr{(u,v)\in \TR(a,b):\sbr{\nt{\nabla u}^2+\nt{\nabla v}^2}^{1/2}<R }.$$
					\bl \label{lemma 2.2}
					The following statements hold.
					{	\begin{itemize}[fullwidth,itemindent=0em]
							\item[(1)] 	$m(a,b)=\inf_{A_{R_0}(a,b)}I(u,v)<0$,
							\item[(2)]	$m(a,b)\le m(a_1,b_1)$ for any $0<a_1\le a$, $0<b_1\le b$.
					\end{itemize}	}
					\el
					\bp
					\begin{itemize}[fullwidth,itemindent=0em]
						\item[(1)]
						From Lemma \ref{structure1}, we have $\PR^+(a,b)\subset A_{R_0}$ and
						$$m(a,b)=\inf_{\PR(a,b)}I(u,v)=\inf_{\PR^+(a,b)} I(u,v)<0.$$
						Obviously $m(a,b)\ge\inf_{A_{R_0}(a,b)}I(u,v)$. On the other hand, for any $(u,v)\in A_{R_0}(a,b)$, there is
						$$m(a,b)\leq I(t_+(u,v)\s(u,v))\le I(u,v).$$
						It follows that $m(a,b)=\inf_{A_{R_0}(a,b)}I(u,v)$.
						\item[(2)] 	Recall the definition of $\rho_*$ in \eqref{rho}, we see that
						\be\label{tem3-6-7}
						m(a,b)=\inf_{A_{\rho_*}(a,b)} I(u,v).
						\ee
						We need only to show that for any arbitrary $\e>0$, one has
						$$m(a,b)\le m(a_1,b_1)+\e.$$
						Let $(u,v)\in A_{\rho_*}(a_1,b_1)$ be such that $I(u,v)\le m(a_1,b_1)+\f{\e}{2}$ and
						$\phi\in\CR_0^\iy(\RN)$ be a cut-off function
						$$0\le\phi\le1\quad\text{and}\quad \phi(x)=\begin{cases}0,\quad\text{if}~|x|\ge2,\\1,\quad\text{if}~|x|\le1. \end{cases}$$
						For any $\delta>0$, we define $u_\delta(x)=u(x)\phi(\delta x)$ and $v_\delta(x)=v(x)\phi(\delta x)$.
						Then $(u_\delta,v_\delta)\to(u,v)$ in $H$ as $\delta\to0$.
						As a consequence, we can fix a $\delta>0$ small enough such that
						\be\label{tem31}
						I(u_\delta,v_\delta)\le I(u,v)+\f{\e}{4}\quad\text{and}\quad
						\sbr{\nt{\nabla u_\delta}^2+\nt{\nabla v_\delta}^2}^{1/2}<\rho_*-\eta,
						\ee
						for a small $\eta>0$. Now we take $\vp\in\CR_0^\iy(\RN)$ such that
						$supp(\vp)\subset B(0,1+\f{4}{\delta})\setminus B(0,\f{4}{\delta})$,
						where $B(0,r)$ is the ball with a radius of $r$ and a center at the origin.
						Set
						$$w_a=\f{\sqrt{a^2-\nt{u_\delta}^2}}{\nt{\vp}}\vp \quad\text{and}\quad w_b=\f{\sqrt{b^2-\nt{u_\delta}^2}}{\nt{\vp}}\vp.$$
						Then for any $\la<0$, noting that
						$$\sbr{supp(u_\delta)\cup supp(v_\delta)}\cap \sbr{supp(\la\s w_a)\cup supp(\la\s w_b)}=\emptyset,$$
						one has $(u_\delta+\la\s w_a,v_\delta+\la\s w_b)\in\TR(a,b)$.
						Since $I(\la\s (w_a,w_b))\to0$ and $\sbr{\nt{\nabla \la\s w_a}^2+\nt{\nabla \la\s w_b}^2}^{1/2}\to0$ as $\la\to-\iy$, we have
						\be\label{tem32}
						I(\la\s (w_a,w_b))\le\f{\e}{4}\quad\text{and}\quad
						\sbr{\nt{\nabla \la\s w_a}^2+\nt{\nabla \la\s w_b}^2}^{1/2}\le \f{\eta}{2}
						\ee
						for $\la<0$ sufficiently close to negative infinity.
						It follows that
						$$\sbr{\nt{\nabla (u_\delta+\la\s w_a)}^2+\nt{\nabla (v_\delta+\la\s w_b)}^2}^{1/2}< \rho_*.$$
						Now using \eqref{tem31}, \eqref{tem32}, we obtain
						$$m(a,b)\le I(u_\delta+\la\s w_a,v_\delta+\la\s w_b)=I(u_\delta,v_\delta)+I(\la\s w_a,\la\s w_b)\le m(a_1,b_1)+\e,$$
						which finish the proof.
					\end{itemize}
					\ep

					\vskip0.12in
					\bp[\bf Proof of the Theorem \ref{thm1} (1) and (3) ]
					First we construct a Palais-Smale sequence for $I_{\TR(a,b)}$ at level $m(a,b)$.	
					By the symmetric decreasing rearrangement, one can easily check that $m(a,b)=\inf_{A_{R_0}(a,b)\cap H_{rad}} I(u,v)$.
					Choosing a minimizing sequence $(\tilde u_n,\tilde v_n)\in A_{R_0}(a,b)\cap H_{rad}$,
					we assume $(\tilde u_n,\tilde v_n)$ are nonnegative by replacing  $(\ti u_n,\ti v_n)$ with $(|\ti u_n|,|\ti v_n|)$.
					Furthermore, using the fact that $I(t_+(\ti u_n,\ti v_n)\s(\ti u_n,\ti v_n))\leq I(\ti u_n,\ti v_n)$,
					we can assume that $(\ti u_n,\ti v_n)\in\PR^+(a,b)\cap H_{rad}$ for $n\ge1$.
					Therefore, by Ekeland's varational principle, there is a radial symmetric
					Palais-Smale sequence $\lbr{(u_n,v_n)}$ for $I|_{\TR(a,b)\cap H_{rad}}$ (hence a Palais-Smale sequence for
					$I|_{\TR(a,b)}$) with the property $||(u_n,v_n)-(\tilde u_n,\tilde v_n)||_H\to0$ as $n\to\infty$, which implies that
					$$P(u_n,v_n)=P(\ti u_n,\ti v_n)+o(1)\to0\quad \text{and}\quad u_n^-,v_n^-\ra0\ \text{a.e. in}\ \RN.$$
					
					Now we want to apply Proposition \ref{PS1} with $c=m(a,b)$, it remains to check conditions \eqref{levelcon1}, \eqref{levelcon2}.
					For that purpose,  since $e(a),e(b)<0$, we only need to prove that $m(a,b)<\min\lbr{e(a),~e(b)}$. 	Without loss of generality, we only prove $m(a,b)<e(a)$. Let $u_0$ be the unique solution in Theorem A.
					From $h(R_0)=0$, we have $\f{1}{2}R_0^2>\f{1}{p}\CR(N,p)a^{p(1-\ga_p)}R_0^{p\ga_p}$, that is
					$R_0^{2-p\ga_p}>\f{2}{p}\CR(N,p)a^{p(1-\ga_p)}$. It follows that
					$$\nt{\nabla u_0}^2=\ga_p\np{u_0}^p\le\ga_p\CR(N,p)a^{p(1-\ga_p)}\nt{\nabla u_0}^{p\ga_p}
					<R_0^{2-p\ga_p}\nt{\nabla u_0}^{p\ga_p},$$
					which gives $\nt{\nabla u_0}<R_0$. Let $v_0=\f{b}{a}u_0$. Then $(u_0,-s\s v_0)\in A_{R_0}(a,b)$ for $s>0$ large enough.
					Therefore for large $s>0$,
					$$m(a,b)\le I(u_0,s\s v_0)\le E(u_0)+\f{e^{-2s}}{2}\nt{\nabla v_0}^2-\f{e^{-p\ga_ps}}{p}\np{v_0}^p<e(a).$$
					
					According to Proposition \ref{PS1}, there exists $u,v\in H_{rad}^1(\RN)$, $u,v>0$ and $\la_1,\la_2>0$ such that
					up to a subsequence $(u_n,v_n)\to(u,v)$ in $H^1(\RN)\times H^1(\RN)$ and
					$(\la_{1,n},\la_{2,n})\to(\la_1,\la_2)$ in $\R^2$. And hence $(u,v)$ is a normalized ground state solution of \eqref{mainequ}.
					\ep

					\vskip 0.3in

					Now we search for the second normalized solution of \eqref{mainequ}  at the mountain pass level of the Energy  functional. Define the energy level
					\be\label{level2}
					l(a,b):=\inf_{\PR^-(a,b)} I(u,v).
					\ee

					\medbreak	
					According to Proposition \ref{PS1}, the first step  of finding  a second critical point for $I|_{\TR(a,b)}$  is to obtain    a Palais-Smale sequence $\lbr{(u_n,v_n)}\subset \TR(a,b)\cap H_{rad}$ for $I|_{\TR(a,b)}$ at
					level $l(a,b)$, with $P(u_n,v_n)\to0$ and $u_n^-,v_n^-\to0$ a.e. in $\RN$ as $n\to\iy$ .

					\medbreak

					For that purpose, we will follow the standard way as in  \cite{Jeanjean=1997}.  Set $\TR_r(a,b)=\TR(a,b)\cap H_{rad}$. Define
					$$\Ga_1:=\lbr{\ga\in C([0,1],\R\times \TR_r(a,b)):~\ga(0)\in(0,\PR^+(a,b)),~\ga(1)\in(0,I^{2m(a,b)})},$$
					where we denote the sublevel set by $I^c:=\lbr{(u,v)\in H:~I(u,v)\le c}$.
					
					\begin{lemma}\label{lemma2.3}
						Let
						\begin{equation}\label{s1}
							 \sigma_1:=\inf_{\ga\in\Ga_1}\max_{t\in[0,1]} J(\ga(t))
						\end{equation}  with $J(s,(u,v)):=I(s\s (u,v))$.
						Then $l(a,b)=\sigma_1$.
					\end{lemma}
					\begin{proof}
						Indeed, for any $(u,v)\in\PR^-(a,b)$,
						let $\ga(t):=\sbr{0, ((1-t)t_+(u,v)+tt_0)\s(u,v)}$ with  $t_0>0$ large enough. So $\ga\in \Ga_1$, and hence
						$$\sigma_1\le \max_{t\in[0,1]} J(\ga(t))\le \max_{t\in\R} \Phi_{(u,v)}(t)=I(u,v),$$
						which gives $\sigma_1\le l(a,b)$. On the other hand, for any $\ga=(\ga_1,\ga_2)\in\Ga_1$, we can assume that $\ga_1=0$.
						Since $\ga_2(0)\in \PR^+(a,b)$ we have $t_-(\ga_2(0))>t_+(\ga_2(0))=0$; since $I(\ga_2(1))\le 2m(a,b)$ we have $t_-(\ga_2(1))<0$
						(in fact, if $t_-(\ga_2(1))\ge0$, then $m(a,b)\le I(\ga_2(1))\le 2m(a,b)$, which is impossible). It follows that there exists
						$\tau_\ga\in(0,1)$ such that $t_-(\ga_2(\tau_\ga))=0$, that is $\ga_2(\tau_\ga)\in\PR^-(a,b)$, and hence
						$$l(a,b)\le I(\ga_2(\tau_\ga))\le \max_{t\in[0,1]} J(\ga(t)),$$
						which gives $l(a,b)\le \sigma_1$. Then we obtain that $l(a,b)=\sigma_1.$    We conclude also that
						\be\label{tem3-5-1}
						\ga_2([0,1])\cap \PR^-(a,b)\neq\emptyset,\quad\forall \ga\in\Ga_1.
						\ee
					\end{proof}
					Now  we can construct a  Palais-Smale sequence at level $l(a,b)$.
					
					\bl\label{PSseq1}
					There exists a Palais-Smale sequence $(u_n,v_n)\subset \TR(a,b)\cap H_{rad}$  for $I|_{\TR(a,b)}$ at
					level $l(a,b)$, with $P(u_n,v_n)\to0$ and $u_n^-,v_n^-\to0$ a.e. in $\RN$ as $n\to\iy$.
					\el
					\bp
					Following  the strategies in \cite[Section 5]{Ghoussoub=1993}, it is easy to check that $\FR=\lbr{A=\ga([0,1]):\ga\in\Ga_1}$
					is a homotopy stable family of compact
					subsets of $X=\R\times\TR_r(a,b)$ with boundary $B=(0, \PR^+(a,b))\cup(0, I^{2m(a,b)})$. Set
					$F=(0,\PR^-(a,b))$, then \eqref{tem3-5-1} gives that
					$$A\cap F\setminus B\neq\emptyset,\quad\forall A\in\FR.$$
					Moreover, we have
					$$\sup I(B)\le0<\sigma_1\le\inf I(F).$$
					So the assumptions $(F'1)$ and $(F'2)$ in \cite[Theorem 5.2]{Ghoussoub=1993} are
					satisfied. Therefore, taking a minimizing
					sequence $\lbr{\ga_n=(0,\beta_n)}\subset\Ga_1$ with $\beta_n\ge0$ a.e. in $\RN$, there exists a Palais-Smale sequence
					$\lbr{(s_n,w_n,z_n)}\subset\R\times \TR_r(a,b)$ for $J|_{\R\times\TR_r(a,b)}$ at level $\sigma_1$, that is
					\be\label{tem8}
					\pa_s J(s_n,w_n,z_n)\to0\quad\text{and}\quad \pa_{(u,v)} J(s_n,w_n,z_n)\to0\quad \text{as }n\to\iy,
					\ee
					with the additional property
					\be\label{tem7}
					|s_n|+\text{dist}_{H}((w_n,z_n),\beta_n([0,1]))\to0 \quad \text{as }n\to\iy.
					\ee
					Let $(u_n,v_n)=s_n\s (w_n,z_n)$.
					The first condition in \eqref{tem8} implies $P(u_n,v_n)\to0$, while the second condition gives
					\be
					\begin{aligned}
						&\quad~~\nm{\rd I|_{\TR(a,b)}(u_n,v_n)}\\
						&=\sup\lbr{|\rd I(u_n,v_n)[(\phi,\psi)]|:~(\phi,\psi)\in T_{(u_n,v_n)}\TR(a,b),\nm{(\phi,\psi)}_H\le1}\\
						&=\sup\lbr{|\rd I(s_n\s (w_n,z_n))[s_n\s (-s_n)\s(\phi,\psi)]|:~(\phi,\psi)\in T_{(u_n,v_n)}\TR(a,b),\nm{(\phi,\psi)}_H\le1}\\
						&=\sup\lbr{|\pa_{(u,v)} J(s_n,w_n,z_n)[(-s_n)\s(\phi,\psi)]|:~(\phi,\psi)\in T_{(u_n,v_n)}\TR(a,b),\nm{(\phi,\psi)}_H\le1}\\
						&\le \nm{\pa_{(u,v)} J(s_n,w_n,z_n)}
						\sup\lbr{\nm{(-s_n)\s(\phi,\psi)]}_H:~(\phi,\psi)\in T_{(u_n,v_n)}\TR(a,b),\nm{(\phi,\psi)}_H\le1}\\
						&\le C \nm{\pa_{(u,v)} J(s_n,w_n,z_n)}\to0\quad \text{as }n\to\iy.
					\end{aligned}
					\ee
					{	Finally, \eqref{tem7} implies that $u_n^-,v_n^-\to0$ a.e. in $\RN$.}
					\ep

					\vskip 0.2in
					According to Proposition \ref{PS1}, the subsequent step in finding a second critical point of $I|_{\TR(a,b)}$ involves establishing the inequality:
					\begin{equation}\label{s2}
						0<l(a,b) <m(a,b)+\f{\nu^{-(N-2)/2}}{N}\SR_{\al,\beta}^{N/2}.  
					\end{equation}
					It is noteworthy that  for any $u\in\PR^-(a,b)$, we can deduce from  Lemma \ref{structure1} that $0$ is the unique maximum point of $\Phi_{(u,v)}$. Consequently,  we obtain
					\be\label{tem4-13-1}
					\begin{aligned}
						I(u,v)&=\max_{t\in\R}I(t\s(u,v))\\&\ge\max_{t\in\R} h(\sbr{\nt{\nabla (t\s u)}^2+\nt{\nabla (t\s v)}}^{1/2})
						=\max_{t\in\R} h(t)>h(\rho_*)>0	,
					\end{aligned}
					\ee
					which means $l(a,b)>0$.
					However, the proof of the second inequality is notably more challenging. We begin by establishing several properties of $l(a,b)$.
					
					\bl\label{tem3-6-5}
					The following statements hold.
					\begin{itemize}[fullwidth,itemindent=0em]
						\item[(1)]	For any $0\le a_1\le a$ and $0\le b_1\le b$, there is $l(a,b)\le l(a_1,b_1)$.
						\item[(2)]	Set
						\be\label{tem3-6-2}
						\Ga(a,b):=\lbr{ \ga(t)\in C([0,1],\bar\TR(c,d)):
							\ga(0)\in \bar V(a,b),~\ga(1)\in I^{2m(a,b)} }.
						\ee
						with
						$$\bar\TR(a,b):=\bigcup_{c\in[a/2,a],d\in[b/2,b]} \TR(c,d),\quad
						\bar V(a,b):=\bigcup_{c\in[a/2,a],d\in[b/2,b]} V(c,d)$$
						and
						$$V(c,d):=\lbr{(u,v)\in\TR(c,d):~I|_{\TR(c,d)}'(u,v)=0,~I(u,v)=m(c,d)}.$$
						Then $l(a,b)\le \inf_{\ga\in\Ga(a,b)}\max_{t\in[0,1]}I(\ga(t))$.
					\end{itemize}
					\el
					\bp
					\begin{itemize}[fullwidth,itemindent=0em]
						\item[(1)] The proof is inspired by   \cite[Lemma 3.2]{Jeanjean-Lu=CVPDE=2020}, here we give the details for the completeness.
						It is sufficient to prove that for any arbitrary $\e>0$, one has
						\be\label{tem2-4}
						l(a,b)\le l(a_1,b_1)+\e.
						\ee
						By the definition of $l(a_1,b_1)$, there exists $(u,v)\in\PR^-(a,b)$ such that
						\be\label{tem2-1}
						I(u,v)\le l(a_1,b_1)+\f{\e}{2}.
						\ee
						Let $\vp\in\CR_0^\iy(\RN)$ be radial and such that
						$$0\le\vp\le1;\quad\vp(x)=1,~\text{when~}|x|\le1; \quad\vp(x)=0~\text{when}~|x|\ge2.$$
						For any small $\delta>0$, we define $u_\delta(x):=u(x)\vp(\delta x)$, $v_\delta(x):=v(x)\vp(\delta x)$.
						Since $(u_\delta,v_\delta)\to(u,v)$ in $H$ as $\delta\to0^+$,
						by Lemma \ref{structure1}, one has $\lim_{\delta\to0^+}t_-(u_\delta,v_\delta)=t_-(u,v)$ and hence
						$$t_-(u_\delta,v_\delta)\s(u_\delta,v_\delta)\to t_-(u,v)\s(u,v)\quad
						\text{in}~H~\text{as}~\delta\to0^+.$$
						As a consequence, we can fix a $\delta>0$ small enough such that
						\be\label{tem2-2}
						I(t_-(u_\delta,v_\delta)\s(u_\delta,v_\delta))\le I(u,v)+\f{\e}{4}.
						\ee
						Now take $\psi\in\CR_0^\iy(\RN)$ such that $supp(\psi)\subset B(0,\f{4}{\delta})^c$ and set
						$$\psi_a=\f{\sqrt{a^2-\nt{u_\delta}^2}}{\nt{\psi}}\psi\quad\text{and}
						\quad\psi_b=\f{\sqrt{b^2-\nt{v_\delta}^2}}{\nt{\psi}}\psi.$$
						Then for any $\tau\le0$, one has
						$$\sbr{supp(u_\delta)\cup supp(u_\delta)}\cap\sbr{supp(\tau\s\psi_a)\cup supp(\tau\s\psi_b)}=\emptyset,$$
						and hence
						$$(\tilde u_\tau,\tilde v_\tau):=(u_\delta+\tau\s\psi_a,v_\delta+\tau\s\psi_b)\in \TR(a,b).$$
						Let $t_\tau:=t_-(\tilde u_\tau,\tilde v_\tau)$ which is given in Lemma \ref{structure1}.
						{From $P(t_\tau\s(\tilde u_\tau,\tilde v_\tau))=0$, i.e.,}
						$$e^{(2-p\ga_p)t_\tau}(\nm{\tilde u_\tau}^2+\nt{\tilde v_\tau}^2)
						=\nm{\tilde u_\tau}_p^p+\nm{\tilde v_\tau}_p^p
						+\nu e^{(2^*-p\ga_p)t_\tau}\int_\RN|\tilde u_\tau|^\al|\tilde v_\tau|^\beta\dx,$$
						we see that $\limsup_{\tau\to-\iy} t_\tau<+\iy$, since $(\tilde u_\tau,\tilde v_\tau)\to(u_\delta,v_\delta)\neq(0,0)$ as
						$\tau\to-\iy$. Then one has that $t_\tau+\tau\to-\iy$ as $\tau\to-\iy$ and that for $\tau<-1$ small enough
						\be\label{tem2-3}
						I((t_\tau+\tau)\s(\psi_a,\psi_b))<\f{\e}{4}.
						\ee
						Now using \eqref{tem2-1}, \eqref{tem2-2} and \eqref{tem2-3}, we obtain
						\begin{align*}
							l(a,b)&\le I(t_\tau\s(\tilde u_\tau,\tilde v_\tau))\\
							&=I(t_\tau\s(u_\delta,v_\delta))+I((t_\tau+\tau)\s(\psi_a,\psi_b))\\
							&\le I(t_\nu(u_\delta,v_\delta)\s(u_\delta,v_\delta))+\f{\e}{4}\\
							&\le I(u,v)+\f{\e}{4}\le l(a_1,b_1)+\e,
						\end{align*}
						which completes the proof.
						
						\vskip 0.1in
						\item[(2)]	Take a $\ga(t)\in\Ga(a,b)$. Since $\ga(0)\in V(c,d)$ for some $c\in[a/2,a]$ and $d\in[b/2,b]$ we have
						$t_-(\ga(0))>t_+(\ga(0))=0$; since $I(\ga(1))<2m(a,b)$ we have $t_-(\ga(1))<0$. So there exists a $t_\ga\in(0,1)$
						such that $t_-(\ga(t_\ga))=0$, i.e., $\ga(t_\ga)\in\PR^-(c,d)$ for some $c\in[a/2,a]$ and $d\in[b/2,b]$. Then
						$$\max_{t\in[0,1]} I(\ga(t))\ge I(\ga(t_\ga))\ge l(c,d)\ge l(a,b),$$
						which completes the proof.
					\end{itemize}
					\ep

					Let
					\be\label{eta}
					\eta_\e(x)=\phi U_\e,
					\ee
					where $U_\e$ is defined by \eqref{soliton} and $\phi\in\CR_0^\iy(\RN)$ is a radial cut-off function
					$$0\le\phi\le1;\quad\phi(x)=1~\text{when~}|x|\le1; \quad\phi(x)=0~\text{when}~|x|\ge2.$$
					Then we have the following estimations, and the  proofs can be found in \cite{Jeanjean=2020,Soave=JFA=2020}.
					\bl\label{est}
					As $\e\to0^+$, we have
					$$\nt{\nabla \eta_\e}^2=\SR^{N/2}+O(\e^{N-2}),\quad \ns{\eta_\e}^{2^*}=\SR^{N/2}+O(\e^N),$$
					$$\nt{\eta_\e}^2=\begin{cases} 	O(\e^2)&\quad \text{if}\quad N\ge5,\\
						O(\e^2|\log \e|) &\quad\text{if}\quad N=4,\\
						O(\e)&\quad\text{if}\quad N=3,		\end{cases}$$
					$$\nm{\eta_\e}_p^p=\begin{cases}  O(\e^{N-(N-2)p/2})&\quad\text{if}\quad N\ge4~\text{and}~p\in(2,2^*)~
						\text{or if}~N=3~\text{and}~p\in(3,6),\\
						O(\e^{p/2})&\quad\text{if}\quad N=3~\text{and}~p\in(2,3),\\
						O(\e^{3/2}|\log\e|)&\quad\text{if} \quad N=3 ~\text{and}~p=3.		\end{cases} $$
					\el
					
					\vskip 0.13 in	
					
					Observe that for $p>2$ and $\al,\beta>1$,
					\be\label{tem4-13-2}
					(1+t)^p\ge1+t^p+pt,\quad\forall t>0,
					\ee
					\be\label{tem4-13-3}
					(1+t_1)^\al(1+t_2)^\beta\ge 1+t_1^\al t_2^\beta,\quad\forall t_1,t_2>0.
					\ee
					Then for any $u,v>0$, $t>0$, there holds
					\be\label{tem3-6-4}
					I(u+t\sqrt\al\eta_\e,v+t\sqrt\beta\eta_\e)
					\le I(u,v)+I(t\sqrt\al\eta_\e,t\sqrt\beta\eta_\e)+t\int_\RN \sbr{\sqrt\al\nabla u+\sqrt\beta\nabla v}\cdot\nabla \eta_\e \dx
					\ee
					
					\bl\label{tem3-6-3}
					For small $\e>0$,
					\begin{itemize}[fullwidth,itemindent=0em]
						\item[(1)]	there exists a $t_*>0$ independent with $\e$ such that
						$\max_{t\ge t_*}I(t\sqrt\al\eta_\e,t\sqrt\beta\eta_\e)+t\le 2m(a,b)$;
						\item[(2)]	$\max_{t>0}I(t\sqrt\al\eta_\e,t\sqrt\beta\eta_\e)
						<\f{\nu^{-(N-2)/2}}{N}\SR_{\al,\beta}^{N/2}+O(\e^{N-2})-C\nm{\eta_\e}_p^p$
					\end{itemize}
					\el
					\bp
					\begin{itemize}[fullwidth,itemindent=0em]
						\item[(1)]	Observe that
						\be\label{tem3-6-1}
						I(t\sqrt\al\eta_\e,t\sqrt\beta\eta_\e)\le \f{\al+\beta}{2}t^2\nt{\nabla\eta_\e}^2
						-\f{\nu\al^{\al/2}\beta^{\beta/2}}{2^*}t^{2^*}\ns{\eta_\e}^{2^*}.
						\ee
						We have that $\nt{\nabla\eta_\e}^2\to\SR^{N/2}$ and $\ns{\eta_\e}^{2^*}\to\SR^{N/2}$, see Lemma \ref{est}.
						Thus there exists a $t_*>0$ independent with $\e$ such that $\max_{t\ge t_*}I(t\sqrt\al\eta_\e,t\sqrt\beta\eta_\e)+t\le 2m(a,b)$.
						
						\vskip 0.1in
						\item[(2)]		In view of \eqref{tem3-6-1}, if $t\to0$, we have
						$I(t\sqrt\al\eta_\e,t\sqrt\beta\eta_\e)<\f{1}{2}\f{\nu^{-(N-2)/2}}{N}\SR_{\al,\beta}^{N/2}$.
						Moreover, for a small $t_0>0$,
						\be
						\begin{aligned}
							\max_{t\ge t_0}I(t\sqrt\al\eta_\e,t\sqrt\beta\eta_\e)
							&\le \max_{t>0}\f{\al+\beta}{2}t^2\nt{\nabla\eta_\e}^2
							-\f{\nu\al^{\al/2}\beta^{\beta/2}}{2^*}t^{2^*}\ns{\eta_\e}^{2^*}-C\nm{\eta_\e}_p^p\\
							&=\f{1}{N}\mbr{ \f{(\al+\beta)\nt{\nabla\eta_\e}^2 }
								{\sbr{\nu\al^{\al/2}\beta^{\beta/2}\ns{\eta_\e}^{2^*}}^{2/2^*}} }^{N/2}-C\nm{\eta_\e}_p^p\\
							&=\f{\nu^{-(N-2)/2}}{N} \sbr{\sbr{\f{\al}{\beta}}^{\beta/2^*}+\sbr{\f{\beta}{\al}}^{\al/2^*}}
							\sbr{\SR+O(\e^{N-2})}^{N/2}-C\nm{\eta_\e}_p^p\\
							&=\f{\nu^{-(N-2)/2}}{N}\SR_{\al,\beta}^{N/2}+O(\e^{N-2})-C\nm{\eta_\e}_p^p.
						\end{aligned}
						\ee
						Thus we conclude that
						$$\max_{t>0}I(t\sqrt\al\eta_\e,t\sqrt\beta\eta_\e)
						<\f{\nu^{-(N-2)/2}}{N}\SR_{\al,\beta}^{N/2}+O(\e^{N-2})-C\nm{\eta_\e}_p^p.$$
					\end{itemize}
					\ep
					
					Now let
					$$a_\e^2:=a^2-2t_*^2\nt{\eta_\e}^2,\quad b_\e^2:=b^2-2t_*^2\nt{\eta_\e}^2.$$
					Clearly $a_\e\to a$, $b_\e\to b$ as $\e\to0$.
					Take $(u_\e,v_\e)\in V(a_\e,b_\e)$ the positive solution obtained in Theorem \ref{thm1} (1).
					Since $u_\e,v_\e$ are radially symmetric, there holds
					\begin{equation}\label{f5}
						\sup_{x\in\RN}\lbr{|u_\e(x)|,|v_\e(x)|}\le C_\e|x|^{-(N-2)/2},
					\end{equation}
					with $C_\e>0$.
					Then using a similar approach as \cite[Lemma 5.5 and  5.6]{Jeanjean=2020}, we can find a sequence $y_\e\in\RN$ such that
					\be
					a_\e\le \nt{u_\e(\cdot-y_\e)+t\sqrt\al\eta_\e}\le a,\quad\forall 0\le t\le t_*,
					\ee
					\be
					b_\e\le \nt{v_\e(\cdot-y_\e)+t\sqrt\beta\eta_\e}\le b,\quad\forall 0\le t\le t_*,
					\ee
					and
					\be\label{tem3-6-6}
					\int_\RN \sbr{\sqrt\al\nabla u_\e(\cdot-y_\e)+\sqrt\beta\nabla v_\e(\cdot-y_\e)}\cdot\nabla \eta_\e\dx
					\le \nt{\eta_\e}^2.
					\ee
					Then we can prove  the following energy estimate.
					\bl\label{keyest}
					If $N\ge4$, we have $$l(a,b) <m(a,b)+\f{\nu^{-(N-2)/2}}{N}\SR_{\al,\beta}^{N/2}.$$
					\el
					\bp
					Let $\ga_\e(t):=(u_\e(\cdot-y_\e)+tt_*\sqrt\al\eta_\e,v_\e(\cdot-y_\e)+tt_*\sqrt\beta\eta_\e)$.
					Clearly $\ga_\e\in\Ga(a,b)$ for $\e>0$ small, see \eqref{tem3-6-2} and Lemma \ref{tem3-6-3}.
					In view of \eqref{tem3-6-4}, \eqref{tem3-6-6}, Lemma \ref{tem3-6-5} and Lemma \ref{tem3-6-3}, we obtain that
					\be\label{tem3-6-11}
					\begin{aligned}
						l(a,b)&\le \max_{t\in[0,1]}I(\ga_\e(t))\le I(u_\e,v_\e)+\max_{t>0}I(t\sqrt\al\eta_\e,t\sqrt\beta\eta_\e)+t_*\nt{\eta_\e}^2\\
						&\le m(a_\e,b_\e)+\f{\nu^{-(N-2)/2}}{N}\SR_{\al,\beta}^{N/2}+O(\e^{N-2})-C\nm{\eta_\e}_p^p+C\nt{\eta_\e}^2
					\end{aligned}
					\ee
					Now we give an upper bound estimate of $m(a_\e,b_\e)$.
					Let $(u_0,v_0)\in V(a,b)$ be the positive solution obtained in Theorem \ref{thm1} (1).
					Let $w_\e=\f{a_\e}{a}u_0$ and $z_\e=\f{b_\e}{b}v_0$.
					Thus $\nt{w_\e}=a_\e$, $\nt{z_\e}=b_\e$ and
					$$\sbr{\nt{\nabla w_\e}^2+\nt{\nabla z_\e}^2}^{1/2}\le \sbr{\nt{\nabla u_0}^2+\nt{\nabla v_0}^2}^{1/2}< \rho_*,$$
					which means $(w_\e,z_\e)\in A_{\rho_*}(a_\e,b_\e)$. From \eqref{tem3-6-7}, it follows that
					\be\label{tem3-6-12}
					\begin{aligned}
						m(a_\e,b_\e)&\le I(w_\e,z_\e)=m(a,b)+I(w_\e,z_\e)-I(u_0,v_0)\\
						&\le m(a,b)+\f{1}{p}\sbr{1-\sbr{\f{a_\e}{a}}^p}\nm{u_0}_p^p+\f{1}{p}\sbr{1-\sbr{\f{b_\e}{b}}^p}\nm{v_0}_p^p\\
						&\quad~~	+\f{\nu}{2^*}\sbr{1-\sbr{\f{a_\e}{a}}^\al\sbr{\f{b_\e}{b}}^\beta}\int_\RN u_0^\al v_0^\beta\dx\\
						&\le m(a,b)+C_1(a^2-a_\e^2)+C_2(b^2-b_\e^2)\\
						&\le m(a,b)+C\nt{\eta_\e}^2.
					\end{aligned}
					\ee
					Combining \eqref{tem3-6-11} and \eqref{tem3-6-12}, we deduce from Lemma \ref{est} that
					\be
					\begin{aligned}
						l(a,b)&\le m(a,b)+\f{\nu^{-(N-2)/2}}{N}\SR_{\al,\beta}^{N/2}+O(\e^{N-2})-C\nm{\eta_\e}_p^p+C\nt{\eta_\e}^2\\
						&=m(a,b)+\f{\nu^{-(N-2)/2}}{N}\SR_{\al,\beta}^{N/2}+O(\e^2)-O(\e^{N-(N-2)p/2})\\
						&<m(a,b)+\f{\nu^{-(N-2)/2}}{N}\SR_{\al,\beta}^{N/2},
					\end{aligned}
					\ee
					where we use the fact $N\ge4$.
					\ep

					\begin{lemma} \label{keyest2}
						
						If $N=3$, we have $$l(a,b) <m(a,b)+\f{\nu^{-(N-2)/2}}{N}\SR_{\al,\beta}^{N/2}.$$
						
					\end{lemma}
					
					\begin{proof}
						By Theorem \ref{thm1} (1), we obtain a normalized ground state solution  $(u_+,v_+)\in \PR^+(a,b)$  of system \eqref{mainequ} with  $\la_{1,+},\la_{2,+}>0$, which is positive and  radially symmetric.  Recall the definition of $\eta_\e$ (see \eqref{eta}), we take
						\begin{equation}
							w_{\varepsilon,t}=u_{+}+t\sqrt{\al}\eta_\e, \quad v_{\varepsilon,t}=v_{+}+t\sqrt{\beta}\eta_\e \quad  \text{ for } t>0.
						\end{equation}
						Now we  define
						\begin{equation}\label{f3}
							W_{\varepsilon,t}(x)=\tau^{\frac{1}{2}}w_{\varepsilon,t}(\tau x), \quad V_{\varepsilon,t}(x)= \xi^{\frac{1}{2}}v_{\varepsilon,t}(\xi x).
						\end{equation}
						Since $N=3$, there holds
						\begin{equation} \label{f1}
							\begin{aligned}
								& \nt{\nabla 	W_{\varepsilon,t}}^2= \nt{\nabla 	w_{\varepsilon,t}}^2, \quad \quad     \nt{\nabla  V_{\varepsilon,t}}^2= \nt{\nabla 	v_{\varepsilon,t}}^2,\\
								&\nt{ 	W_{\varepsilon,t}}^2= \tau^{-2}\nt{ 	w_{\varepsilon,t}}^2, \quad  \quad \nt{  V_{\varepsilon,t}}^2= \xi^{-2} \nt{  	v_{\varepsilon,t}}^2,\\
								&\np{ 	W_{\varepsilon,t}}^p= \tau^{p\ga_p-p}\np{ 	w_{\varepsilon,t}}^p, \quad \np{  V_{\varepsilon,t}}^p= \xi^{p\ga_p-p} \np{  	v_{\varepsilon,t}}^p.
							\end{aligned}
						\end{equation}	
						By choosing
						\begin{equation}\label{f4}
							\tau =\cfrac{\nt{ 	w_{\varepsilon,t}} }{a}\geq 1, \quad \xi=\cfrac{\nt{ 	v_{\varepsilon,t}} }{b} \geq 1,
						\end{equation}
						we obtain that $(	W_{\varepsilon,t},  V_{\varepsilon,t}) \in\TR(a,b)$. Then  by Lemma \ref{structure1}, there exists a unique $k_{\varepsilon,t}  >0$ such that
						\begin{equation}
							\sbr{\overline{W}_{\varepsilon,t}, \overline{V}_{\varepsilon,t}}	:=\sbr{	k_{\varepsilon,t}^{\frac{3}{2}} 	W_{\varepsilon,t}(k_{\varepsilon,t}x), k_{\varepsilon,t}^{\frac{3}{2}} 	V_{\varepsilon,t}(k_{\varepsilon,t}x)  } \in \mathcal{P}^-(a,b),
						\end{equation}
						which implies that
						\begin{equation} \label{f2}
							k_{\varepsilon,t}^2\sbr{ \nt{\nabla 	W_{\varepsilon,t}}^2 +\nt{\nabla 	V_{\varepsilon,t}}^2}= \ga_{p}k_{\varepsilon,t}^{p\ga_{p}} \sbr{ \np{ 	W_{\varepsilon,t}}^p+\np{ 	V_{\varepsilon,t}}^p} +\nu k_{\varepsilon,t}^{2^*}\int_{\R^3}|W_{\varepsilon,t}|^\alpha|V_{\varepsilon,t}|^\beta\dx.
						\end{equation}
						Since  $(u_+,v_+)\in \PR^+(a,b)$  and  it follows from Lemma \ref{structure1}  that   $k_{\varepsilon,0} >1$. On the other hand, since $2<p<2+\frac{4}{N}$ and $0<p\ga_{p}<2$,  one can deduce from  \eqref{f1}, \eqref{f2} that $k_{\varepsilon,t}\to 0$ as  $t\to \infty$. By using Lemma  \ref{structure1} again,   $k_{\varepsilon,t}$ is continuous for $t$,  which implies that there exists  $t_\e>0$ such that $k_{\varepsilon,t_\e}=1$ for $\varepsilon$ small enough.  Hence,
						\begin{equation}\label{f12}
							\begin{aligned}
								l(a,b)=\inf_{(u,v)\in\PR^-(a,b)}I(u,v) \leq I(\overline{W}_{\varepsilon,t_\e}, \overline{V}_{\varepsilon,t_\e})=I(W_{\varepsilon,t_\e}, V_{\varepsilon,t_\e}) \leq \sup_{t>0}I(W_{\varepsilon,t }, V_{\varepsilon,t }).
							\end{aligned}
						\end{equation}
						Note that  $(u_+,v_+)$  is the    normalized ground state solution   of system \eqref{mainequ}  and $\eta_\e$ is a  positive function, one can deduce from \eqref{f1} and Lemma \ref{est} that there exists $t_0>0$ such that
						\begin{equation}\label{f11}
							\begin{aligned}
								I(W_{\varepsilon,t }, V_{\varepsilon,t })\leq m(a,b)+\f{\nu^{-(N-2)/2}}{N}\SR_{\al,\beta}^{N/2} - \tilde{\delta}
							\end{aligned}
						\end{equation}
						for $t<t_0^{-1}$ and $t>t_0$ with $\tilde{\delta} >0$.
						
						On the other hand,
						for   $t_0^{-1} \leq t \leq t_0$,    by \eqref{tem4-13-3}, \eqref{f3} and  \eqref{f1}  we have
						\begin{equation}\label{f7}
							\begin{aligned}
								I(W_{\varepsilon,t }, V_{\varepsilon,t })&=\frac{1}{2}\sbr{\nt{\nabla 	w_{\varepsilon,t}}^2 +\nt{\nabla 	v_{\varepsilon,t}}^2}-\frac{1}{p}\tau^{p\ga_p-p}\np{ 	w_{\varepsilon,t}}^p-\frac{1}{p}\xi^{p\ga_p-p}\np{ 	v_{\varepsilon,t}}^p\\
								&-\frac{\nu}{2^*}\tau^{\frac{\alpha}{2}}\xi^{\frac{\beta}{2}}\int_{\R^3}|w_{\varepsilon,t}(\tau x)|^\alpha|v_{\varepsilon,t}(\xi x)|^\beta \mathrm{d} x\\
								&\leq I(u_+,v_+)+\f{\al+\beta}{2}t^2\nt{\nabla\eta_\e}^2
								-\f{\nu\al^{\al/2}\beta^{\beta/2}}{2^*}t^{2^*}\ns{\eta_\e}^{2^*}-\frac{\al^{\frac{p}{2}}+\beta^{\frac{p}{2}}}{p}t^p\nm{\eta_\e}_p^p\\
								&+t\int_{\R^3}(\sqrt{\alpha}\nabla u_+ + \sqrt{\beta}\nabla v_+)\cdot \nabla \eta_\e+\frac{1-\tau^{p\ga_{p}-p}}{p}\np{ 	w_{\varepsilon,t}}^p+\frac{1-\xi^{p\ga_{p}-p}}{p}\np{ 	v_{\varepsilon,t}}^p \\
								&+\frac{\nu}{2^*}\sbr{ \int_{\R^3}|w_{\varepsilon,t}(  x)|^\alpha|v_{\varepsilon,t}(  x)|^\beta \mathrm{d} x-\tau^{\frac{\alpha}{2}}\xi^{\frac{\beta}{2}}\int_{\R^3}|w_{\varepsilon,t}(\tau x)|^\alpha|v_{\varepsilon,t}(\xi x)|^\beta \mathrm{d} x  }\\
								&=:I_1+I_2+\frac{\nu}{2^*} I_3.
							\end{aligned}
						\end{equation}
						For  $t_0^{-1} \leq t \leq t_0$, since $I(u_+,v_+)=m(a,b)$, we  deduce from Lemma \ref{est} and \eqref{Sobolev2} that
						\begin{equation}\label{f8}
							\begin{aligned}
								I_1&\leq I(u_+,v_+)+ \sup_{t>0}\sbr{ \f{\al+\beta}{2}t^2\nt{\nabla\eta_\e}^2
									-\f{\nu\al^{\al/2}\beta^{\beta/2}}{2^*}t^{2^*}\ns{\eta_\e}^{2^*}}-C \nm{\eta_\e}_p^p\\
								&\leq m(a,b)+\f{\nu^{-(N-2)/2}}{N}\SR_{\al,\beta}^{N/2}.
							\end{aligned}
						\end{equation}
						Note that  $(u_+,v_+)$ is a radially symmetric solution  of \eqref{mainequ} and decays  to zero as $r\to \infty$, one has
						\begin{equation}
							\int_{\R^3} u_+ \eta_\e \dx \sim \varepsilon^{\frac{5}{2}}\int_{1}^{\frac{1}{\varepsilon}}\sbr{\frac{1}{1+r^2}}^{\frac{1}{2}}r^2 ~\mathrm{d}r\sim \varepsilon^{\frac{1}{2}}, ~~ \text{ and } \quad 	\int_{\R^3} v_+ \eta_\e \dx\sim\varepsilon^{\frac{1}{2}}.
						\end{equation}
						Then we have
						\begin{equation}\label{f6}
							\begin{aligned}
								&\tau^2=1+\frac{2\sqrt{\al}t}{a^2}	\int_{\R^3} u_+ \eta_\e\dx+\frac{t^2\al}{a^2}\int_{\R^3}|\eta_\e|^2\dx=1+O(\varepsilon^{\frac{1}{2}}),\\
								&\xi^2=1+\frac{2\sqrt{\beta}t}{b^2}	\int_{\R^3} v_+ \eta_\e\dx+\frac{t^2\beta}{b^2}\int_{\R^3}|\eta_\e|^2\dx=1+O(\varepsilon^{\frac{1}{2}}).
							\end{aligned}
						\end{equation}	
						Since  $(u_+,v_+) $ is a   solution   of system \eqref{mainequ} with  $\la_{1,+},\la_{2,+}>0$,  it follows from a similar argument as used in \cite[Lemma 5.5]{Jeanjean=2020} that
						\begin{equation}
							\begin{aligned}
								\int_{\R^3} \nabla u_+ \cdot \nabla  \eta_\e\dx &=-\la_{1,+}	\int_{\R^3}   u_+    \eta_\e \dx+ \int_{\R^3} |u_+|^{p-2}u_+ \eta_\e \dx +\frac{\alpha\nu}{2^*}\int_{\R^3}|u_+|^{\al-2}|v_+|^{\beta}u_+ \eta_\e\dx\\
								&	\leq -\la_{1,+}	\int_{\R^3}   u_+    \eta_\e \dx+ C\nt{\eta_\e}^2 =-\la_{1,+}	\int_{\R^3}   u_+  \eta_\e \dx +O(\varepsilon).
							\end{aligned}
						\end{equation}	
						Similarly, we have
						\begin{equation}
							\int_{\R^3} \nabla v_+ \cdot \nabla  \eta_\e\dx \leq -\la_{2,+}	\int_{\R^3}   v_+  \eta_\e \dx +O(\varepsilon).
						\end{equation}
						Therefore, combining these with \eqref{f6} we have
						\begin{equation}
							\begin{aligned}	
								I_2&=	t\int_{\R^3}(\sqrt{\alpha}\nabla u_+ + \sqrt{\beta}\nabla v_+)\cdot \nabla \eta_\e\dx+\frac{1-\tau^{p\ga_{p}-p}}{p}\np{ 	w_{\varepsilon,t}}^p+\frac{1-\xi^{p\ga_{p}-p}}{p}\np{ 	v_{\varepsilon,t}}^p\\
								&\leq \frac{t\sqrt{\al}}{a^2}\sbr{ -\la_{1,+}a^2	\int_{\R^3}   u_+    \eta_\e \dx +(1-\ga_{p})\np{ 	w_{\varepsilon,t}}^p\int_{\R^3}   u_+    \eta_\e\dx}\\
								&+\frac{t\sqrt{\beta}}{b^2}\sbr{ -\la_{2,+}b^2	\int_{\R^3}   v_+    \eta_\e\dx +(1-\ga_{p})\np{ 	v_{\varepsilon,t}}^p\int_{\R^3}   v_+    \eta_\e\dx}+O(\varepsilon).
							\end{aligned}
						\end{equation}	
						Note that the fact $	 \la_{1,+}a^2+  \la_{2,+}b^2=(1-\ga_{p})\sbr{\np{u_+}^p+\np{v_+}^p}$ which comes from the Pohozaev identity  $P(u_+,v_+)=0$, then it is easy to verify that
						\begin{equation}\label{f9}
							I_2=t\int_{\R^3}(\sqrt{\alpha}\nabla u_+  + \sqrt{\beta}\nabla v_+)\cdot \nabla \eta_\e\dx+\frac{1-\tau^{p\ga_{p}-p}}{p}\np{ 	w_{\varepsilon,t}}^p+\frac{1-\xi^{p\ga_{p}-p}}{p}\np{ 	v_{\varepsilon,t}}^p= O(\varepsilon).
						\end{equation}
						Now we estimate
						\begin{equation}
							I_3=\int_{\R^3}|w_{\varepsilon,t}(  x)|^\alpha|v_{\varepsilon,t}(  x)|^\beta \mathrm{d} x-\tau^{\frac{\alpha}{2}}\xi^{\frac{\beta}{2}}\int_{\R^3}|w_{\varepsilon,t}(\tau x)|^\alpha|v_{\varepsilon,t}(\xi x)|^\beta \mathrm{d} x .
						\end{equation}	
						Note that $\tau, \xi \geq 1$,
						\begin{equation}
							\begin{aligned}
								I_3&=(1-\tau^{\frac{\al}{2}}\xi^{\frac{\beta}{2}})\int_{\R^3}|w_{\varepsilon,t}(  x)|^\alpha|v_{\varepsilon,t}(  x)|^\beta \mathrm{d} x\\
								&+\tau^{\frac{\al}{2}}\xi^{\frac{\beta}{2}}\sbr{ \int_{\R^3}|w_{\varepsilon,t}(  x)|^\alpha|v_{\varepsilon,t}(  x)|^\beta \mathrm{d} x -\int_{\R^3}|w_{\varepsilon,t}(  \tau x)|^\alpha|v_{\varepsilon,t}(  x)|^\beta \mathrm{d} x}\\
								&+\tau^{\frac{\al}{2}}\xi^{\frac{\beta}{2}}\sbr{ \int_{\R^3}|w_{\varepsilon,t}(  \tau x)|^\alpha|v_{\varepsilon,t}(  x)|^\beta \mathrm{d} x -\int_{\R^3}|w_{\varepsilon,t}(  \tau x)|^\alpha|v_{\varepsilon,t}(  \xi x)|^\beta \mathrm{d} x}\\
								& \leq \tau^{\frac{\al}{2}}\xi^{\frac{\beta}{2}}(1-\tau)   \int_{\R^3}\al |w_{\varepsilon,t}(  x)|^{\alpha-1}|v_{\varepsilon,t}(  x)|^\beta\sbr{ x\cdot \nabla  w_{\varepsilon,t}(  x)}\mathrm{d} x  +o(|1-\tau|)\\
								&+\tau^{\frac{\al}{2}}\xi^{\frac{\beta}{2}}(1-\xi)   \int_{\R^3}\beta |w_{\varepsilon,t}(  \tau x)|^{\alpha}|v_{\varepsilon,t}(  x)|^{\beta-1}\sbr{ x\cdot \nabla v_{\varepsilon,t}(  x)}\mathrm{d} x  +o(|1-\xi|).
							\end{aligned}
						\end{equation}	
						For small $\varepsilon>0$ and $t_0^{-1}<t<t_0$, we may assume that $\tau,\xi \in \mbr{1,2}$. Moreover, one can easily check that
						\begin{equation}
							\int_{\R^3}|\eta_\e|^{2^*}\dx=	\int_{\R^3}|\eta_1|^{2^*}\dx, \quad  \int_{\R^3}|x\cdot \nabla \eta_\e|^{2^*}\dx=\int_{\R^3}|x\cdot \nabla \eta_1|^{2^*}\dx,
						\end{equation}
						Hence, with the help of H\"older inequality, it is easy to see that
						\begin{equation}
							\int_{\R^3}\al |w_{\varepsilon,t}(  x)|^{\alpha-1}|v_{\varepsilon,t}(  x)|^\beta\sbr{ x\cdot \nabla  w_{\varepsilon,t}(  x)}\mathrm{d} x \quad \text{ and } \quad \int_{\R^3}\beta |w_{\varepsilon,t}(  \tau x)|^{\alpha}|v_{\varepsilon,t}(  x)|^{\beta-1}\sbr{ x\cdot \nabla v_{\varepsilon,t}(  x)}\mathrm{d} x
						\end{equation}
						are bounded independent of $\varepsilon$. By \eqref{f6}, there exists a constant $K>0$ independent of $\varepsilon$ such that
						\begin{equation}\label{f10}
							I_3\leq -K\sbr{\int_{\R^3}   u_+    \eta_\e\dx +\int_{\R^3}   v_+    \eta_\e\dx }+o(\varepsilon^{\frac{1}{2}})\leq -K\varepsilon^{\frac{1}{2}}+o(\varepsilon^{\frac{1}{2}}).
						\end{equation}
						Therefore, by \eqref{f7},\eqref{f8},\eqref{f9},\eqref{f10}, we have 			
						\begin{equation}
							\begin{aligned}
								I(W_{\varepsilon,t }, V_{\varepsilon,t })&\leq  m(a,b)+\f{\nu^{-(N-2)/2}}{N}\SR_{\al,\beta}^{N/2}-K\varepsilon^{\frac{1}{2}}+o(\varepsilon^{\frac{1}{2}})+O(\varepsilon)\\
								&<m(a,b)+\f{\nu^{-(N-2)/2}}{N}\SR_{\al,\beta}^{N/2}
							\end{aligned}
						\end{equation}
						for $\varepsilon$ small enough and $t_0^{-1}<t<t_0$. Then it follows from 		\eqref{f11} that
						\begin{equation}
							\sup_{t>0}	I(W_{\varepsilon,t }, V_{\varepsilon,t }) <m(a,b)+\f{\nu^{-(N-2)/2}}{N}\SR_{\al,\beta}^{N/2}.
						\end{equation}
						Finally, the conclusion follows directly from \eqref{f12}. This completes the proof.

					\end{proof}

					\bp[\bf Proof of Theorem \ref{thm1} (2) and (3)]
					With the help of Lemma \ref{PSseq1}, Lemma \ref{tem3-6-5},  Lemma \ref{keyest}  and Lemma \ref{keyest2},
					the conclusion is just a combinition of Proposition \ref{PS1} and above lemmas.
					\ep

					\subsection{The case \texorpdfstring{$2+\f{4}{N}\leq p<2^*$}{}}\label{Sect2.2}
					
					In this subsection, we consider the case $2+\f{4}{N}\le p<2^*$ and assume that \eqref{H3-2} hold.
					Following the strategies in  \cite{Soave=JDE=2020}, we have the following lemma.
					\bl\label{structure2}
					For every $(u,v)\in \TR(a,b)$, $\Phi_{(u,v)}(t)$ has exactly one critical point $t_-(u,v)$. Moreover,
					\begin{itemize}[fullwidth,itemindent=0em]
						\item[(a)]	$\PR(a,b)=\PR^-(a,b)$; \item[(b)]$t\s(u,v)\in\PR(a,b)$ if and only if $t=t_-(u,v)$,
						\item[(c)] 	$\Phi_{(u,v)}(t)$ is strictly dereasing and concave on $(t_-(u,v),+\iy)$ and
						$$\Phi_{(u,v)}(t_-(u,v))=\max_{t\in\R}\Phi_{(u,v)}(t)>0,$$
						\item[(d)] 	the map $(u,v)\mapsto t_-(u,v)$ is of class $C^1$.
					\end{itemize}
					\el
					
					Moreover, in a standard way, we can prove
					\bl\label{PSseq2}
					There exists a Palais-Smale sequence $(u_n,v_n)\subset \TR(a,b)\cap H_{rad}$ for $I|_{\TR(a,b)}$ at
					level $m(a,b)$, with $P(u_n,v_n)\to0$ and $u_n^-,v_n^-\to0$ a.e. in $\RN$ as $n\to\iy$.
					\el
					\bp
					Noting that $m(a,b)=\inf_{\ga\in\Ga_2}\max_{t\in[0,1]} J(\ga(t))$
					with $J(s,(u,v)):=I(s\s (u,v))$ and
					$$\Ga_2:=\lbr{\ga\in C([0,1],\R\times \TR_r(a,b)):~\ga(0)\in(0,A_k(a,b)),~\ga(1)\in(0,I^0)},$$
					for a small $k>0$. Then such a required Palais-Smale sequence can be found as in Lemma \ref{PSseq1}.
					For more details, we refer to \cite{Bartsch-Jeanjean-Soave=JMPA=2016}.
					\ep
					
					Now in view of Proposition \ref{PS1}, to obtain the compactness of the Palais-Smale sequence, one need
					some fine estimations about $m(a,b)$, which will be done in the remainder of this section.
					
					\bl\label{est3-1}
					Let $N=3,4$. Then
					\begin{itemize}[fullwidth,itemindent=0em]
						\item[(1)]	for any $0\le a_1\le a$ and $0<b_1\le b$, there is $m(a,b)\le m(a_1,b_1)$;
						\item[(2)]	there holds $0<m(a,b)<\f{\nu^{-(N-2)/2}}{N}\SR_{\al,\beta}^{N/2}$.
					\end{itemize}
					\el
				
					\bp[Proof of Lemma \ref{est3-1}]
					\begin{itemize}[fullwidth,itemindent=0em]
						\item[(1)]	The proof is just the same as that of Lemma \ref{tem3-6-5} (1).
						\item[(2)]	It is natural  that $m(a,b)>0$. Indeed, for any $(u,v)\in\PR(a,b)$, there is
						$$\begin{aligned}
							\nt{\nabla u}^2+\nt{\nabla v}^2&=\ga_p(\nm{u}_p^p+\nm{v}_p^p)+\nu\int_\RN|u|^\al|v|^\beta\dx\\
							&\le C_1(\nt{\nabla u}^2+\nt{\nabla v}^2)^{\f{p\ga_p}{2}}+C_2(\nt{\nabla u}^2+\nt{\nabla v}^2)^{\f{2^*}{2}}\\
						\end{aligned}$$
						which implies $\inf_{\PR(a,b)}\nt{\nabla u}^2+\nt{\nabla v}^2 >0$. So we have
						$$\begin{aligned}
							m(a,b)&=\inf_{\PR(a,b)}I(u,v)=\inf_{\PR(a,b)} I(u,v)-\f{1}{p\ga_p}P(u,v)\\
							&=\f{p\ga_p-2}{2p\ga_p}(\nt{\nabla u}^2+\nt{\nabla v}^2)+\f{2^*-p\ga_p}{2^*p\ga_p}\nu\int_\RN|u|^\al|v|^\beta\dx\\
							&\ge C\inf_{\PR(a,b)}\nt{\nabla u}^2+\nt{\nabla v}^2 >0.
						\end{aligned}$$
						
						Now let $\eta_\e$ be defined by \eqref{eta}. Take $u_\e=\f{\sqrt\al c}{\nt{\eta_\e}}\eta_\e$ and
						$v_\e=\f{\sqrt\beta  c}{\nt{\eta_\e}}\eta_\e$ for a small constant $c>0$.
						Let $t_\e:=t_-(u_\e,v_\e)$ be given by Lemma \ref{structure2}. So $t_\e\s(u_\e,v_\e)\in\PR(\sqrt\al c,\sqrt\beta c)$
						and hence for a proper $c>0$ we obtain
						\be\label{tem3-9-5}
						\begin{aligned}
							m(a,b)&\le m(\sqrt\al c,\sqrt\beta c)\le I(t_\e\s(u_\e,v_\e))\\
							&=\f{1}{2}e^{2t_\e}\sbr{\nt{\nabla u_\e}^2+\nt{\nabla v_\e}^2}-\f{\nu}{2^*}e^{2^*t_\e}\int_\RN u_\e^\al v_\e^\beta\dx
							-\f{1}{p}e^{p\ga_p t_\e}\sbr{\nm{u_\e}_p^p+\nm{v_\e}_p^p}\\
							&\le \max_{s>0}\sbr{\f{\al+\beta}{2}s^2\nt{\nabla\eta_\e}^2
								-\f{\nu\al^{\al/2}\beta^{\beta/2}}{2^*}s^{2^*}\ns{\eta_\e}^{2^*}}-C\nm{\eta_\e}_p^p\nt{\eta_\e}^{-p}e^{p\ga_pt_\e}\\
							&=\f{1}{N}\mbr{ \f{(\al+\beta)\nt{\nabla\eta_\e}^2 }
								{\sbr{\nu\al^{\al/2}\beta^{\beta/2}\ns{\eta_\e}^{2^*}}^{2/2^*}} }^{N/2}-C\nm{\eta_\e}_p^p\nt{\eta_\e}^{-p}e^{p\ga_pt_\e}\\
							&=\f{\nu^{-(N-2)/2}}{N}\SR_{\al,\beta}^{N/2}+O(\e^{N-2})-C\nm{\eta_\e}_p^p\nt{\eta_\e}^{-p}e^{p\ga_pt_\e}.
						\end{aligned}
						\ee
						
						We claim that $e^{t_\e}\ge C\nt{\eta_\e}$ for some constant $C>0$ as $\e\to0$. Now
						we check it. If $p=2+\f{4}{N}$, by  definition $P(t_\e\s(u_\e,v_\e))=0$ and hence
						\be\label{tem3-9-1}
						\begin{aligned}
							e^{(2^*-2)t_\e}&=\f{\nt{\nabla u_\e}^2+\nt{\nabla v_\e}^2}{\nu\int_\RN u_\e^\al v_\e^\beta\dx}
							-\ga_{p}\f{\nm{u_\e}_p^p+\nm{v_\e}_p^p}{ \nu\int_\RN u_\e^\al v_\e^\beta \dx }\\
							&\ge \sbr{1-\f{2\CR(N,p)}{p}(a^{4/N}+b^{4/N})}\f{\nt{\nabla u_\e}^2+\nt{\nabla v_\e}^2}{\nu\int_\RN u_\e^\al v_\e^\beta \dx}\\
							&\ge C \f{\nt{\nabla \eta_\e}^2\nt{\eta_\e}^{2^*-2}}{\ns{\eta_\e}^{2^*}},
						\end{aligned}
						\ee
						where we used \eqref{H3-2}. By Lemma \ref{est}, we have $\nt{\nabla \eta_\e}^2\sim 1$ and $\ns{\eta_\e}^{2^*}\sim1$,
						which in turn together with \eqref{tem3-9-1} gives that $e^{t_\e}\ge C\nt{\eta_\e}$.
						If $p>2+\f{4}{N}$, by  definition $P(t_\e\s(u_\e,v_\e))=0$ and hence
						$$e^{(2^*-2)t_\e}\nu\int_\RN u_\e^\al v_\e^\beta\dx\le\nt{\nabla u_\e}^2+\nt{\nabla v_\e}^2,$$
						whence it follow that
						\be\label{tem3-9-2}
						e^{t_\e}\le\sbr{\f{\nt{\nabla u_\e}^2+\nt{\nabla v_\e}^2}{\nu\int_\RN u_\e^\al v_\e^\beta\dx}}^{1/(2^*-2)}.
						\ee
						By \eqref{tem3-9-2} and using the fact that $p\ga_p>2$,
						\be\label{tem3-9-3}
						\begin{aligned}
							e^{(2^*-2)t_\e}&=\f{\nt{\nabla u_\e}^2+\nt{\nabla v_\e}^2}{\nu\int_\RN u_\e^\al v_\e^\beta\dx}
							-\ga_{p}\f{\nm{u_\e}_p^p+\nm{v_\e}_p^p}{ \nu\int_\RN u_\e^\al v_\e^\beta \dx }e^{(p\ga_p-2)t_\e}\\
							&\ge \f{\nt{\nabla u_\e}^2+\nt{\nabla v_\e}^2}{\nu\int_\RN u_\e^\al v_\e^\beta\dx}
							-\ga_{p}\f{\nm{u_\e}_p^p+\nm{v_\e}_p^p}{ \nu\int_\RN u_\e^\al v_\e^\beta\dx }
							\sbr{\f{\nt{\nabla u_\e}^2+\nt{\nabla v_\e}^2}{\nu\int_\RN u_\e^\al v_\e^\beta \dx}}^{\f{p\ga_p-2}{2^*-2}}\\
							&=\f{\nt{\nabla\eta_\e}^2\nt{\eta_\e}^{2^*-2}}{\ns{\eta_\e}^{2^*}}
							\sbr{C_1-C_2\nt{\nabla\eta_\e}^{-2\f{2^*-p\ga_p}{2^*-2}}
								\ns{\eta_\e}^{2^*\f{p\ga_p-2}{2^*-2}}\f{\nm{\eta_\e}_p^p}{\nt{\eta_\e}^{p(1-\ga_p)}}  }.
						\end{aligned}
						\ee
						Using Lemma \ref{est}, we obtain $\nt{\nabla \eta_\e}^2\sim 1$, $\ns{\eta_\e}^{2^*}\sim1$ and that
						\be\label{tem3-9-4}
						\f{\nm{\eta_\e}_p^p}{\nt{\eta_\e}^{p(1-\ga_p)}}=\begin{cases}
							O(\e^{\f{6-p}{4}})\quad&\text{if}~N=3,\\  O(|\log \e|^{-\f{p(1-\ga_p)}{2}})\quad&\text{if}~N=4.
						\end{cases}
						\ee
						Going back to \eqref{tem3-9-3}, it results that $e^{t_\e}\ge C\nt{\eta_\e}$.
						
						Substituting $e^{t_\e}\ge C\nt{\eta_\e}$ into \eqref{tem3-9-5}, we obtain
						$$m(a,b)\le\f{\nu^{-(N-2)/2}}{N}\SR_{\al,\beta}^{N/2}+O(\e^{N-2})-C\f{\nm{\eta_\e}_p^p}{\nt{\eta_\e}^{p(1-\ga_p)}},$$
						and hence using \eqref{tem3-9-4} we infer that $m(a,b)<\f{\nu^{-(N-2)/2}}{N}\SR_{\al,\beta}^{N/2}$ for any $\e>0$ small,
						which is the desired result.
					\end{itemize}
					\ep
					
					By Lemma  \ref{est3-1}, we  observe that the energy level $m(a,b)$ satisfies \eqref{levelcon2} because $e(a)>0, e(b)>0$.
					However, we still need to show that $m(a,b)\neq e (a)$ and $m(a,b)\neq e (b)$. Actually,   there holds
					\bl\label{est3-2}
					For any $a,b>0$,  the following energy estimate holds
					\begin{equation}
						m(a,b)< \min\lbr{e(a),~e(b)},
					\end{equation}
					under one of the following conditions:
					\begin{itemize}[fullwidth,itemindent=0em]
						\item[(1)]  $a\leq a_0$ and  $b\leq a_0$, when $2+\f{4}{N}<p<2^*$, where $a_0$ is defined in \eqref{defi of a0};
						\item[(2)]	$\nu>\nu_1~$ for some  $\nu_1=\nu_1(a,b,\alpha,\beta)>0$;
						\item[(3)]  $\al<2$ and  $a\leq b$;
						\item[(4)]  $\beta<2$ and $b\leq a$.
					\end{itemize}
					\el
					
					\begin{proof}
						The proofs of (2)-(4) in Lemma \ref{est3-2} is very similar as that of \cite[Lemma 4.6]{Li-Zou-1}, so we omit the details here.	
						Now we give a simple proof for     (1).	 	
						Recalling the definition of $a_0$ (see \eqref{defi of a0}), by the properties of $e(a)$ in Theorem A, $a_0$ is actually choosed to satisfy
						\[e(a_0)=\f{\nu^{-(N-2)/2}}{N}\SR_{\al,\beta}^{N/2}.\]
						Note that  $e(a)$ is strictly decreasing with respect to $a>0$,     we can see from Lemma \ref{est3-1} that  $$	m(a,b)< \min\lbr{e(a),~e(b)},$$
						since $a\leq a_0$ and  $b\leq a_0$. This completes the proof.
					\end{proof}
					
					\begin{remark}
						{\rm  	We point out that the conclusions of Lemma \ref{est3-1} and  \ref{est3-2} still hold true when $N\ge5$. }
					\end{remark}
					
					\begin{proof}[\bf Proof of Theorem \ref{thm2}]
						With the help of Lemma \ref{PSseq2},  \ref{est3-1} and  \ref{est3-2},  the conclusion is just a combinition of Proposition \ref{PS1} and above lemmas.
					\end{proof}

					\vskip0.23in
					\section{Existence for the repulsive case  \texorpdfstring{$\nu<0$}{} } \label{Sect3}
					In this section, we study the existence of normalized solution of \eqref{mainequ} for the repulsive case  $\nu<0$. Throughout this section, we always work under the following assumptions
					
					\begin{equation} \label{A1}	
						N=3,~ 2<p<2+\frac{4}{N}, \text{ and  }~	 p \leq 	\alpha,\beta \leq 2^*, \alpha+\beta=2^* .
					\end{equation}
					
					Now we give the following compactness result for $\nu<0$.
					\begin{proposition}\label{PS2}
						Assume that  \eqref{A1} holds and
						\be\label{monotonicity v<0}
						m_r(a,b)\le m_r(a_1,b_1)\quad\text{for any}~0<a_1\le a,~0<b_1\le b.
						\ee
						Let $\{(u_n,v_n)\}\subset \TR_r(a,b)$ be a sequence consisting of radial symmetric functions such that
						\be\label{tem3-1 v<0}
						I'(u_n,v_n)+\la_{1,n}u_n+\la_{2,n}v_n\to0\quad \text{for some}~\la_{1,n},\la_{2,n}\in\R,
						\ee
						\be\label{tem3-2 v<0}
						I(u_n,v_n)\to c,\quad P(u_n,v_n)\to0,
						\ee
						\be\label{tem3-3 v<0}
						u_n^-,v_n^-\to0,~\text{a.e. in}~\RN,
						\ee
						with
						\be\label{levelcon2 v<0}
						c<\min \lbr{e(a),~e(b)} \text{ and } c\neq 0.
						\ee
						Then there exists $u,v\in H_{rad}^1(\RN)$ with $u,v>0$  and $\lambda_1,\lambda_2> 0$ such that
						up to a subsequence $(u_n,v_n)\to(u,v)$ in $H^1(\RN)\times H^1(\RN)$ and
						$(\la_{1,n},\la_{2,n})\to(\la_1,\la_2)$ in $\R^2$.
					\end{proposition}
					
					\begin{proof}
						We divide the proof into three steps.
						\vskip 0.1in
						Step 1.) We prove that $\lbr{(u_n,v_n)}$ is bounded in $H^1(\RN)\times H^1(\RN) $.
						\vskip 0.1in
						By $P(u_n,v_n)\to 0$, one can obtain that for $n$ large enough,
						\begin{align*}
							c+1&\ge I(u_n,v_n)-\f{1}{2^*}P(u_n,v_n)\\
							&=\f{1}{N}\sbr{ \nt{\nabla u_n}^2+\nt{\nabla v_n}^2} -\f{2^*-p}{2^*p}\sbr{\np{u_n}^p+\np{v_n}^p}\\
							&\ge \f{1}{N}\sbr{ \nt{\nabla u_n}^2+\nt{\nabla v_n}^2}-C\sbr{ \nt{\nabla u_n}^2+\nt{\nabla v_n}^2}^{p\ga_p/2},
						\end{align*}
						for some $C>0$, which implies that $\{(u_n,v_n)\}$ is bounded because $p\ga_{p}<2$.

						Moreover, from
						$$\la_{1,n}=-\f{1}{a}I'(u_n,v_n)[(u_n,0)]+o(1)\quad\text{and}\quad
						\la_{2,n}=-\f{1}{b}I'(u_n,v_n)[(0,v_n)]+o(1), $$
						we know that $\la_{1,n},\la_{2,n}$ are also bounded.
						So there exists $u,v\in H^1(\RN)$, $\la_1,\la_2\in\R$ such that up to a subsequence
						$$(u_n,v_n)\rh(u,v)\quad\text{in} ~H^1(\RN)\times H^1(\RN),$$
						$$(u_n,v_n)\ra(u,v)\quad\text{in} ~L^q(\RN)\times L^q(\RN),~\text{for}~2<q<2^*,$$
						$$(u_n,v_n)\ra(u,v)\quad\text{a.e. in}~\RN,$$
						$$(\la_{1,n},\la_{2,n})\to(\la_1,\la_2)\quad\text{in} ~\R^2.$$
						Then \eqref{tem3-1 v<0} and \eqref{tem3-3 v<0} give that
						\be\label{tem3-4 v<0}
						\left\{ 	\begin{aligned}
							&I'(u,v)+\la_1 u+\la_2v=0,\\
							&u\ge0,v\ge0,
						\end{aligned}\right.
						\ee
						and hence $P(u,v)=0$.
						\vskip 0.1in
						
						Step 2.)  We prove that $u\neq 0$, $v\neq 0$  and  $\lambda_1,\lambda_2> 0$.
						\vskip 0.1in
						Without loss of generality,  we assume that $u=0$ by contradiction, then    two cases will occur.
						For the case  $v=0$, then from $P(u_n,v_n)=o(1)$  we obtain that	$ \nt{\nabla u_n}^2+\nt{\nabla v_n}^2=\nu\int_\RN |u_n|^\al|v_n|^\beta+\dx+o(1)$. Note that $\nu \leq 0$, we have 	$\nt{\nabla u_n}^2+\nt{\nabla v_n}^2\to 0$, that is, $u_n\to0 $ and $v_n \to 0$ strongly in $\DR^{1,2}(\RN)$. Then
						\begin{equation}
							c=\lim\limits_{n\to \infty} I(u_n,v_n)=\lim\limits_{n\to \infty} \frac{1}{N} \sbr{\nt{\nabla u_n}^2+\nt{\nabla v_n}^2}=0,
						\end{equation}
						which contradicts to the fact $c\neq 0$.  For the case $v\neq 0$, then by maximum principle we know that $v$ is a solution of
						\be
						\left\{ 	\begin{aligned}
							&-\dl v+\la_2 v=v^{p-1},\quad\text{in}~\RN, \\
							&v>0.
						\end{aligned}\right.
						\ee
						and  by using  \cite[Lemma A.2]{Ikoma} and \cite[Theorem 8.4]{Souplet}, we obtain $\la_2>0$. Let $\bar v_n=v_n-v$. Similar to Proposition \ref{PS1},  we obtain that
						\begin{equation}
							\nt{\nabla u_n}^2+\nt{\nabla \bar v_n}^2=\nu\int_\RN |u_n|^\al|\bar v_n|^\beta\dx+o(1),
						\end{equation}
						therefore $\nt{\nabla u_n}^2+\nt{\nabla \bar v_n}^2 \to 0 $ because $\nu <0$, which implies that $u_n\to0 $ and $v_n \to v$ strongly in $ \DR^{1,2}(\RN)$. Then
						$$\begin{aligned}
							&\quad~~\nt{\nabla\bar v_n}^2+\la_2\nt{\bar v_n}^2\\
							&=\sbr{I'(u_n,v_n)+\la_{1,n}u_n+\la_{2,n}v_n}[(0,\bar v_n)]-\sbr{I'(u,v)+\la_1u+\la_{2}v}[(0,\bar v_n)]+o(1)\\
							&=o(1).\end{aligned}$$
						That is $v_n\to v$ in $H^1(\RN)$. Therefore,
						
						\begin{equation}
							c=\lim\limits_{n\to \infty} I(u_n,v_n)=\lim\limits_{n\to \infty} \frac{1}{N} \sbr{\nt{\nabla u_n}^2+\nt{\nabla v_n}^2}+E(v)= e(b),
						\end{equation}
						which is impossible because $c<\min\lbr{e(a),e(b)}$. Then we have $u\neq 0$ and $v\neq 0$. Hence by the  maximum principle,  $u>0$, $v>0$.
						
						Since $\nu<0$,  in order to show that $\lambda_1,\lambda_2>0$, we will borrow some ideas from   \cite[Lemma A.2]{Ikoma}.                                                                 By contradiction, we may assume that $\lambda_1\leq 0$.
						It follows from \cite[Corollary B.1]{Li-Zou-1} that $(u,v)$ is a  smooth solution, and belongs to $L^{\infty}(\RN)\times L^{\infty}(\RN)$, thus $|\Delta u|, |\Delta v| \in L^{\infty}(\RN)$. A standard   gradient estimates for the Poisson  equation (see \cite{G-T=1983}) shows that $|\nabla u|, |\nabla v|\in L^{\infty}(\RN)$. Then from   $u,v\in L^2(\RN)$,  we get $u(x), v(x) \to 0$, as $|x|\to \infty$.
						Recalling the assumption \eqref{A1}, we have
						\begin{equation}
							\begin{aligned}
								-\Delta u= |\lambda_1|u +(1+\frac{\alpha \nu}{2^*}|u|^{\alpha-p}|v|^\beta )|u|^{p-2}u\geq 0
							\end{aligned}
						\end{equation}
						for $|x|>R_0$ with $R_0>0$ large enough, so $u$ is superharmonic at infinity. From the Hadamard three spheres theorem \cite[Chapter 2]{Hadamard three sphere theorem}, we can see that the function $m(r):=\min_{|x|=r} u(x)$ satisfies
						\begin{equation}
							m(r)\geq \cfrac{m(r_1)(r^{2-N}-r_2^{2-N})+m(r_2)(r_1^{2-N}-r^{2-N})}{r_1^{2-N}-r_2^{2-N}}, \quad \text{ for } R_0\leq r_1<r<r_2.
						\end{equation}
						Since $u(x) \to 0$ as $|x|\to \infty$, we can see that $m(r_2)\to 0 $ as $r_2\to \infty$ and  $r^{N-2} m(r) \geq R_0^{N-2}m(R_0)$ for all $r\geq R_0$. Note that  $N=3$,  we have
						\begin{equation}
							\nt{u}^2 \geq C \int_{R_0}^{+\infty} |m(r)|^2 r^{N-1} {\rm d}  r \geq C \int_{R_0}^{+\infty} r^{ 3-N} {\rm d}  r =+\infty
						\end{equation}
						for some $C>0$, which is impossible because $\nt{u}=a$. Therefore, $\lambda_1>0$. Similarly, we can prove that $\lambda_2> 0$.
						
						\vskip 0.1in
						Step 3.) We prove the strong convergence.
						\vskip 0.1in
						Let $(\bar u_n,\bar v_n)=(u_n-u,v_n-v)$. Similar to before, we can show that $u_n\to u$ and $v_n \to v$ strongly in $ \DR^{1,2}(\RN)$. Note that $\lambda_1,\lambda_2>0$  and
						\begin{equation}
							\begin{aligned}
								&\quad~~\nt{\nabla\bar u_n}^2+\la_1\nt{\bar u_n}^2+\nt{\nabla\bar v_n}^2+\la_2\nt{\bar v_n}^2\\
								&=\sbr{I'(u_n,v_n)+\la_{1,n}u_n+\la_{2,n}v_n}[(\bar u_n,\bar v_n)]-\sbr{I'(u,v)+\la_1u+\la_{2}v}[(\bar u_n,\bar v_n)]+o(1)\\
								&=o(1),\end{aligned}
						\end{equation}
						we obtain $(u_n,v_n)\to(u,v)$ strongly in $H^1(\RN)\times H^1(\RN)$.
						This completes the proof.		
					\end{proof}

					
					For any $(u,v)\in\TR_r(a,b)$,	we present an important result about $\Phi_{(u,v)}:\R\to\R$
					$$\Phi_{(u,v)}(t):=\f{1}{2}e^{2t}\int_\RN \sbr{|\nabla u|^2+|\nabla v|^2}\dx
					-\f{1}{p}e^{p\ga_pt}\int_{\RN} \sbr{|u|^p+|v|^p}\dx-\f{\nu}{2^*}e^{2^*t}\int_{\RN}|u|^\al|v|^\beta\dx.$$

					\begin{lemma} \label{lemma v<0 subcritical}
						For all $\nu\leq 0$ and $(u,v)\in\TR_r(a,b)$, there exists a unique $t{(u,v)}\in \R$ such that $t{(u,v)}\s (u,v) \in \PR_r(a,b)$; $t{(u,v)}$ is also the unique critical point of $\Phi_{(u,v)}(t)$ and a strict minimum point at a negative level. Moreover, we have
						\begin{itemize}[fullwidth,itemindent=0em]
							\item[(a)]  $\PR_r(a,b)=\PR_r^+(a,b)$; $t\s(u,v)\in\PR_r(a,b)$ if and only if $t=t(u,v)$;
							\item[(b)]  $\Phi_{(u,v)}(t)$ is strictly decreasing on $ (-\infty,t(u,v))$, and  strictly increasing and convex on $ (t(u,v),+\infty)$, and	$$\Phi_{(u,v)}(t(u,v))=\min_{t\in\R}\Phi_{(u,v)}(t);$$
							\item[(c)]  The map $ (u,v) \mapsto$  $t{(u,v)} $  is of class $C^1$.
						\end{itemize}
						
					\end{lemma}
					
					\begin{proof}
						Observe that
						\begin{equation}
							\begin{aligned}
								\Phi'_{(u,v)}(t)&=e^{2t}\mbr{ \int_\RN \sbr{|\nabla u|^2+|\nabla v|^2}\dx-\ga_pe^{(p\ga_p-2)t}\int_{\RN} \sbr{|u|^p+|v|^p}\dx+|\nu| e^{(2^*-2)t}\int_{\RN}|u|^\al|v|^\beta\dx} \\
								& =: e^{2t}G_{(u,v)}(t),
							\end{aligned}
						\end{equation}
						one can easily see that $G_{(u,v)}(t)$ is strictly increasing   because $p\gamma_p<2<2^*$. 	 	
						A direct computation shows that $G_{(u,v)}(-\infty)=-\infty$ and $G_{(u,v)}(+\infty)=+\infty$, then $G_{(u,v)}$ has a unique zero $t(u,v)$, which is also a unique zero of $\Phi'_{(u,v)}$.   And we can see that  $t\s(u,v)\in\PR_r(a,b)$ if and only if $t=t(u,v)$. Moreover, $\Phi'_{(u,v)}(t)<0$ for $t<t(u,v)$ and $\Phi'_{(u,v)}(t)>0$ for $t>t(u,v)$, which implies that  $t(u,v)$ is a  strict minimum point of $\Phi_{(u,v)} $ at a negative level.
						For any $(u,v)\in \PR_r(a,b)\setminus\PR_r^+(a,b) $, there hold
						\begin{equation}
							\int_\RN \sbr{|\nabla u|^2+|\nabla v|^2}\dx-\ga_p\int_{\RN} \sbr{|u|^p+|v|^p}\dx+|\nu| \int_{\RN}|u|^\al|v|^\beta\dx=0,
						\end{equation}
						and
						\begin{equation}
							2\int_\RN \sbr{|\nabla u|^2+|\nabla v|^2}\dx-\ga_p p\ga_{p}\int_{\RN} \sbr{|u|^p+|v|^p}\dx+2^*|\nu| \int_{\RN}|u|^\al|v|^\beta\dx\leq 0.
						\end{equation}
						Therefore, we have
						\begin{equation}
							0\geq (2-p\ga_{p})\ga_{p}\int_{\RN} \sbr{|u|^p+|v|^p}\dx+(2^*-2)|\nu| \int_{\RN}|u|^\al|v|^\beta\dx,
						\end{equation}
						which is impossible because  $p\ga_{p}<2<2^*$. Consequently              $\PR_r(a,b)=\PR_r^+(a,b)$.   Finally,   the implicit function theorem implies that  the map $ (u,v) \mapsto$  $t{(u,v)} $  is of class $C^1$.
					\end{proof}
					
					\begin{lemma}\label{lemma 3.2}
						For all $\nu\leq 0$ and $2<p<2+\frac{4}{N}$,  we have the following.
						\begin{itemize}[fullwidth,itemindent=0em]
							\item[(a)]  For $0<a_1\leq a$ and $0<b_1\leq b
							$, then $$m_r(a,b)\leq m_r(a_1,b_1)+m_r(a-a_1,b-b_1)\leq  m(a_1,b_1);$$
							\item[(b)]   For all $a,b>0$, we have $$m_r(a,b)\leq e(a)+e(b)<\min\lbr{e(a),e(b)},$$ where $e(a),e(b) $ are given by \eqref{defi of e(a)}.
						\end{itemize}
					\end{lemma}
					\begin{proof}
						(a) By the density of $C_0^\infty(\RN)$ into $H^1(\RN)$,  for arbitrary $\varepsilon>0$, we can choose radial symmetric functions $(u_1,v_1) $, $(u_2,v_2) \in C_0^\infty(\RN) \times C_0^\infty(\RN)$  with $\nt{u_1}=a_1,\nt{v_1}=b_1$ and $\nt{u_2}=a-a_1,\nt{v_2}=b-b_1$ such that
						\begin{equation}
							\begin{aligned}
								&I(u_1,v_1)\leq m_r(a_1,b_1)+\frac{\varepsilon}{2},\\
								&I(u_2,v_2)\leq m_r(a-a_1,b-b_1)+\frac{\varepsilon}{2}.
							\end{aligned}
						\end{equation}
						Observe that $I(u,v)$ is invariant by translation,  then we may assume that  $supp(u_1)\cap supp(u_2)=\emptyset$, $supp(u_1)\cap supp(v_2)=\emptyset$, $supp(v_1)\cap supp(u_2)=\emptyset$ and $supp(v_1)\cap supp(v_2)=\emptyset$, which implies that  $\nt{u_1+u_2}=a$ and $\nt{v_1+v_2}=b$. By Lemma \ref{lemma v<0 subcritical}, there exists a unique strict minimum point $t=t(u_1+u_2,v_1+v_2)\in \R$ such that $t\s (u_1+u_2,v_1+v_2)\in \PR_r(a,b)$ and therefore
						\begin{equation}
							\begin{aligned}
								m(a,b)&\leq I(t\s (u_1+u_2,v_1+v_2)) \leq I( u_1+u_2,v_1+v_2)\\
								&=I( u_1 ,v_1 )+I(  u_2, v_2)\leq m_r(a_1,b_1)+m_r(a-a_1,b-b_1)+\varepsilon
							\end{aligned}
						\end{equation}
						Then by the arbitrariness of $\varepsilon$, we obtain that $m_r(a,b)\leq m_r(a_1,b_1)+m_r(a-a_1,b-b_1)$.  Finally, since $2<p<2+\frac{4}{N}$, by Lemma \ref{lemma v<0 subcritical} we know that  $m_r(a-a_1,b-b_1)\leq 0$ for $0<a_1\leq a$ and $0<b_1\leq b$, and then $m_r(a,b)\leq m_r(a_1,b_1) $.

						\vskip 0.1in

						(b) For $2<p<2+\frac{4}{N}$, it follows from \cite{Soave=JDE=2020} that $e(a)$ and $e(b)$  are attained, then for arbitrary $\varepsilon>0$ we can  take two radial symmetric functions $u_0,v_0\in C_0^\infty(\RN) $ with $\nt{u_0}=a$ and $\nt{v_0}=b$ such that
						\begin{equation}
							E(u_0)\leq e(a)+\varepsilon, \quad 	E(v_0)\leq e(b)+\varepsilon.
						\end{equation}
						We can further assume  that  the supports of $u_0$ and $v_0$ are disjoint, then $I(u_0,v_0)=E(u_0)+E(v_0)$. Since $(u_0,v_0)\in \TR_r(a,b)$, by Lemma \ref{lemma v<0 subcritical}, there exists  $t{(u_0,v_0)}\in \R$ such that $t{(u_0,v_0)}\s (u_0,v_0) \in \PR_r(a,b)$ and $\Phi_{(u_0,v_0)}(t(u_0,v_0))=\min_{t\in\R}\Phi_{(u_0,v_0)}(t).$ Hence, we have
						\begin{equation}
							\begin{aligned}
								m(a,b)&=\inf_{(u,v)\in\PR(a,b)} I(u,v) \leq I(t{(u_0,v_0)}\s (u_0,v_0)) \\
								&=\Phi_{(u_0,v_0)}(t(u_0,v_0)) \leq \Phi_{(u_0,v_0)}(0)\\
								&=I(u_0,v_0)=E(u_0)+E(v_0)\\&\leq e(a)+e(b)+2\varepsilon .
							\end{aligned}
						\end{equation}
						Then  we  obtain $	m_r(a,b) \leq  e(a)+e(b)$ from the arbitrariness of $\varepsilon$. Finally, since $2<p<2+\frac{4}{N}$,  it follows from   \cite{Soave=JDE=2020}  that  $ e(a)<0 $ and $e(b)<0$, and then $m_r(a,b)< \min\lbr{e(a),e(b)} $. This completes the proof.
					\end{proof}

					\begin{proof}[\bf Proof of Theorem \ref{thm3}]
						
						Since $2<p<2+\frac{4}{N}$ and $\nu<0$, one can easily check that $-\infty <m_r(a,b)<0$, then  we take a radial minimizing sequence $\lbr{(\hat u_n, \hat v_n) } \in \PR_r(a,b)  $ such that
						$	I(\hat u_n,\hat v_n)=m(a,b)+o(1).$ Then passing to $\sbr{ | \hat u_n|,| \hat v_n|}$, we may assume that $(\hat u_n, \hat v_n)$  is  nonnegative. By using the Ekeland  variational
						principle (see \cite{Ekeland varitional principle}), we can take a   radially symmetric Palais-Smale sequence $(u_n,v_n) $ for $I|_{\TR_r(a,b)} $ such that $\left\| (u_n,v_n) -(\hat u_n,\hat v_n)\right\|_{H_{rad}} \to 0	$ as $n\to \infty$. Then
						\begin{equation}
							P(u_n,v_n)=P(\hat u_n, \hat v_n)+o(1), \quad \text{ as } n\to \infty.
						\end{equation}	
						By Lagrange multiple rules,  there exist $\lambda_{1,n},\lambda_{2,n}\in \R$ such that  $	I'(u_n,v_n)+\la_{1,n}u_n+\la_{2,n}v_n\to0$ as $n\to \infty$. Then  combining Lemma \ref{lemma 3.2} and Proposition \ref{lemma v<0 subcritical} for  $c=m_r(a,b)$, we obtain that  $(u_n,v_n)\to(u,v)$ in $  H_{rad}$ and
						$(\la_{1,n},\la_{2,n})\to(\la_1,\la_2)$ in $\R^2$. By the strong convergence, $(u_n,v_n) \in \PR_r(a,b)$ is a  normalized  solution of \eqref{mainequ}.
					\end{proof}
					
					\section{Nonexistence  for the defocusing case \texorpdfstring{$\omega_{1}<0$}{} and \texorpdfstring{$\omega_2<0$}{}} \label{Sect4}
					In this section, we will prove the Nonexistence results of normalized solution Theorem \ref{thm5}.
					
					\begin{proof}[\bf Proof of Theorem \ref{thm5}]
						Assume by contradiction that the system  \eqref{mainequ}  has a  positive normalized solution  $(u,v)\in H^1(\RN)\times H^1(\RN) $ with $\nt{u}^2=a^2,~\nt{v}^2=b^2$.  Then from the Pohozaev identity, we obtain that
						\begin{equation}
							\lambda_1 a^2+\lambda_2b^2=(1-\ga_{p})\sbr{ \omega_{1}\np{u}^p+ \omega_{2}\np{v}^p}<0,
						\end{equation}
						since $\omega_1,\omega_2<0$, $\ga_{p}<1$.  Then one of $\lambda_1$ and $\lambda_2$ is strictly less than zero. Without loss of generality, we may assume that $\lambda_1<0$. By a similar argument as used in that of  Step 2 in Proposition \ref{PS2},  we have $u(x), v(x) \to 0$, as $|x|\to \infty$, which together with assumption \eqref{f13} implies
						\begin{equation}
							-\Delta u= (|\lambda_1| + |u|^{p-2}+\frac{\alpha \nu}{2^*}|u|^{\alpha-2}|v|^\beta )u\geq  \frac{|\lambda_1| }{2}u>0.
						\end{equation}
						for $|x|>R_0$ with $R_0>0$ large enough, so $u$ is superharmonic at infinity. Then using the Hadamard three spheres theorem \cite[Chapter 2]{Hadamard three sphere theorem}, we can proceed with the argument presented in Step 2 of  Proposition \ref{PS2} that $r^{N-2} m(r) \geq R_0^{N-2}m(R_0)$ for all $r\geq R_0$, where $m(r):=\min_{|x|=r} u(x)$.
						Then
						\begin{equation}\label{f14}
							\left\|u \right\|_{q}^q \geq C \int_{R_0}^{+\infty} |m(r)|^q r^{N-1} {\rm d}  r \geq C \int_{R_0}^{+\infty} r^{ q(2-N)+N-1} {\rm d}  r
						\end{equation}
						for some $C>0$. {If $N=3,4$,  we choose $q=2$ in  \eqref{f14} and obtain   that
							\begin{equation}
								\left\|u \right\|_{2}^2\geq C \int_{R_0}^{+\infty} r^{ 3-N} {\rm d}  r=+\infty,
							\end{equation}
							which is a contradiction. If $N\geq 5$, it can be observed that $q $ can not be chosen as 2. In order to get a contradiction, we need  choose $q\leq N/(N-2)$. However, $u\in H^1(\RN)$ could not directly imply that $u \in L^q(\RN)$ for some $0<q\leq N/(N-2) $. Therefore,  we need  such additional assumption to get a contradiction. This completes the proof.}
					\end{proof}

					\section{The Sobolev critical case \texorpdfstring{$p=\al+\beta=2^*$}{} }	 \label{Sect5}
					In this section, we  consider the system \eqref{mainequ} with  Sobolev critical  exponent $p=\al+\beta=2^*$ and  prove the existence Theorem  \ref{thm6}  and nonexistence results Theorem \ref{thm4}.
					
					\begin{proof}[\bf Proof of Theorem \ref{thm6}]
						Recalling the definition of $	\Phi_{(u,v)}(t)$ in \eqref{fiber},   for any $\sbr{u,v}\in \TR(a,b)$, we have
						\be
						\begin{aligned}
							\Phi_{(u,v)}(t)=I(t\s (u,v))=&\f{1}{2}e^{2t}\sbr{\nt{\nabla u}^2+\nt{\nabla v}^2}
							-\f{1}{2^*}e^{2^* t} \sbr{ \left\| u\right\|_{2^*}^{2^*}+\left\| v\right\|_{2^*}^{2^*}    +\nu \int_{\RN}|u|^\al|v|^\beta \dx},
						\end{aligned}
						\ee
						Obviously, the function $	\Phi_{(u,v)}(t)$  has a unique critical point $t(u,v)\in \R$, which is also the unique  strict maximum point, such that $t(u,v)\s (u,v)\in \mathcal{P}(a,b)$. More precisely,
						\begin{equation}
							e^{ t(u,v)}=\sbr{\frac{ \nt{\nabla u}^2+\nt{\nabla v}^2}{ \left\| u\right\|_{2^*}^{2^*}+\left\| v\right\|_{2^*}^{2^*}    +\nu \int_{\RN}|u|^\al|v|^\beta\dx}}^{\frac{1}{2^*-2}}.
						\end{equation}
						We claim that $\mathcal{P}^+(a,b)\cup \mathcal{P}^0(a,b)=\emptyset$, implying that  $\mathcal{P}(a,b)=\mathcal{P}^-(a,b)$. Indeed, for any $(u,v)\in\mathcal{P}^+(a,b)\cup \mathcal{P}^0(a,b) $, we have
						\begin{equation}
							2\sbr{\nt{\nabla u}^2+\nt{\nabla v}^2}\geq 2^*\sbr{\left\| u\right\|_{2^*}^{2^*}+\left\| v\right\|_{2^*}^{2^*}    +\nu \int_{\RN}|u|^\al|v|^\beta \dx}=2^*\sbr{\nt{\nabla u}^2+\nt{\nabla v}^2},
						\end{equation}
						implying that $\nt{\nabla u}^2+\nt{\nabla v}^2=0$. which is a contradiction to   $(u,v)\in \TR(a,b)$. Therefore, we prove that the claim is true and
						\begin{equation}
							m(a,b)= \inf_{(u,v)\in\PR(a,b)} I(u,v)=\inf_{(u,v)\in\PR^-(a,b)} I(u,v) .
						\end{equation}
						Moreover, it is standard to show that
						\begin{equation}
							\begin{aligned}
								m(a,b)&=\inf_{(u,v)\in\PR^-(a,b)} I(u,v) = \inf_{(u,v)\in\TR(a,b)} \max_{t\in\R} I(t\s (u,v))	\\
								&=\inf_{(u,v)\in\TR(a,b)} I(t(u,v)\s (u,v))\\
								&=\frac{1}{N}\inf_{(u,v)\in\TR(a,b)} \frac{ \sbr{\nt{\nabla u}^2+\nt{\nabla v}^2}^{\frac{2^*}{2^*-2}}}{\sbr{ \left\| u\right\|_{2^*}^{2^*}+\left\| v\right\|_{2^*}^{2^*}    +\nu \int_{\RN}|u|^\al|v|^\beta \dx}^{\frac{2
										}{2^*-2}}     }.
							\end{aligned}
						\end{equation}
						Now we claim that $ m(a,b)=\mathcal{C} $, where $\mathcal{C}$ is  defined in \eqref{defi of C}.  In fact,  we deduce from \eqref{defi of C t} that $\mathcal{C}\leq	m(a,b) $. On the other hand,  it follows from \cite{HeYang2018} that the ground state  level $\mathcal{C}$  can be achieved by
						\begin{equation}\label{f21}
							\sbr{\widetilde{U}_{\e,y},\widetilde{V}_{\e,y}}=\sbr{ \tilde{x}_1 F_{max}^{-\frac{N-2}{4}}U_{\e,y}, \tilde{x}_2 F_{max}^{-\frac{N-2}{4}}U_{\e,y}}\in \mathcal{D},
						\end{equation}  which is given by \eqref{f15}. Observe that the Aubin-Talenti bubble $U_{\e,y}  $ belongs to $L^2(\RN)$ if and only if   $N\geq 5$
						Then for the unique choice of $\varepsilon_0>0$, we have $\nt{\widetilde{U}_{\e_0,y}}=a$ and $\nt{\widetilde{V}_{\e_0,y}}=b$ since we assume that  $N\geq 5$, $\tilde{x}_1\neq 0$, $\tilde{x}_2\neq 0$ and $a\tilde{x}_2=b\tilde{x}_1$. Therefore,  by \eqref{defi of C t}, \eqref{f22} and  \eqref{f21}
						\begin{equation}
							\begin{aligned}
								m(a,b) &\leq   \frac{ \sbr{\nt{\nabla \widetilde{U}_{\e,y}}^2+\nt{\nabla \widetilde{V}_{\e,y}}^2}^{\frac{2^*}{2^*-2}}}{N\sbr{ \left\| \widetilde{U}_{\e,y}\right\|_{2^*}^{2^*}+\left\| \widetilde{V}_{\e,y}\right\|_{2^*}^{2^*}    +\nu \int_{\RN}|\widetilde{U}_{\e,y}|^\al|\widetilde{V}_{\e,y}|^\beta\dx}^{\frac{2
										}{2^*-2}}     }=\frac{1}{N}F_{max}^{-\frac{N-2}{2}}\mathcal{S}^{\frac{N}{2}}=\mathcal{C}.\\
							\end{aligned}
						\end{equation}
						Therefore, the claim is true and $ m(a,b)=\mathcal{C}$. That is, the minimization of $I(u,v)$ on $\mathcal{P}(a,b)$ is equivalent to the minimization of  $I(u,v)$ on $\mathcal{N}$. It follows from \eqref{f21} that the minimization of  $I(u,v)$ on $\mathcal{N}$ is  $\sbr{\widetilde{U}_{\e_0,y},\widetilde{V}_{\e_0,y}}$ for the unique choice of $\varepsilon_0>0$ such that $\nt{\widetilde{U}_{\e_0,y}}=a$ and $\nt{\widetilde{V}_{\e_0,y}}=b$. Hence the system \eqref{mainequ} has a normalized ground state of the form   $\sbr{\widetilde{U}_{\e_0,y},\widetilde{V}_{\e_0,y}}$ solving the system \eqref{mainequ} with $\lambda_1=\lambda_2=0$. It completes the proof.

					\end{proof}

					\begin{proof}[\bf Proof of Theorem \ref{thm4}]
						Here, we give a brief proof of Theorem 	\ref{thm4}. Indeed, let $(u,v)$ be a positive normalized solution of \eqref{mainequ} with $N=3,4$ and $p=\al+\beta=2^*$,
						then using Pohozaev identity, one can easily obtain that
						\be\label{tem3-21-1}
						\la_1 \nt{u}^2+\la_2 \nt{v}^2=0;
						\ee
						since $u,v$ satisfy
						$$-\Delta u+\la_1u\ge 0\quad\text{and}\quad -\Delta v+\la_2v\ge0,$$
						by \cite[Lemma A.2]{Ikoma}  we obtain immediately $\la_1,\la_2>0$.
						It follows from \eqref{tem3-21-1} that $u=v=0$, which is a contradiction.
					\end{proof}

							\vs
					\noindent {\bf Acknowledgement} \quad			
					
				  Houwang Li is supported by the postdoctoral foundation of BIMSA. 	Tianhao Liu is supported by the Postdoctoral Fellowship Program
						of CPSF  (GZB20240945).  Wenming Zou is supported by  National Key R\&D Program of China (Grant 2023YFA1010001) and NSFC(12171265).

					\vs
					\noindent {\bf Data availability statement} \quad 
					
					No data was used for the research described in the article.

					\vs
					\noindent {\bf Conflict of interest statement } \quad

					On behalf of all authors, the corresponding author states that there is no conflict of interest.


					
				\end{document}